\documentclass[reqno,12pt]{amsart}

\usepackage[active]{srcltx}
\usepackage{color}
\usepackage{amsmath,amsthm,amssymb}
\usepackage{epsfig}
\usepackage{graphicx}
\usepackage{amssymb}
\usepackage{latexsym}
\usepackage{pdfpages}
\usepackage{mathtools}
\usepackage{enumitem}

\usepackage{ulem}

\usepackage{ifpdf}

\mathtoolsset{showonlyrefs}

\def\z{{\bf z}}

\def\divi{\hbox{\rm div\,}}

\ifpdf 
\usepackage[hidelinks]{hyperref}
\else 
\fi

\usepackage[usenames,dvipsnames]{pstricks}
\usepackage{epsfig}
\usepackage{pst-grad} 
\usepackage{pst-plot} 

\usepackage{color}

\newtheorem{theorem}{Theorem}[section]
\newtheorem{lemma}[theorem]{Lemma}
\newtheorem{definition}[theorem]{Definition}
\newtheorem{proposition}[theorem]{Proposition}
\newtheorem{corollary}[theorem]{Corollary}
\newtheorem{remark}[theorem]{Remark}

\newtheorem*{theorem*}{\it Theorem}

\def\divi{\hbox{\rm div\,}}

\def\R{\mathbb R}

\numberwithin{equation}{section}

\def\1{\raisebox{2pt}{\rm{$\chi$}}}

\def\XXint#1#2#3{{\setbox0=\hbox{$#1{#2#3}{\int}$}
		\vcenter{\hbox{$#2#3$}}\kern-.5\wd0}}

\newcommand{\twopartdef}[4]
{
	\left\{
	\begin{array}{ll}
		#1 & #2 \\
		#3 & #4
	\end{array}
	\right.
}

\newcommand{\threepartdef}[6]
{
	\left\{
	\begin{array}{lll}
		#1 & \mbox{if } #2 \\
		#3 & \mbox{if } #4 \\
		#5 & \mbox{if } #6
	\end{array}
	\right.
}

\makeatletter
\newcommand{\labeltext}[2]{%
  \@bsphack
  \csname phantomsection\endcsname 
  \def\@currentlabel{#1}{\label{#2}}%
  \@esphack
}
\makeatother

\begin{document}
	
\title[Strongly anisotropic Anzellotti pairings]{\bf Strongly anisotropic Anzellotti pairings \\ and their applications to the anisotropic $p$-Laplacian}
	
\author[W. G\'{o}rny]{Wojciech G\'{o}rny}
	
\address{ W. G\'{o}rny: Faculty of Mathematics, Universit\"at Wien, Oskar-Morgerstern-Platz 1, 1090 Vienna, Austria; Faculty of Mathematics, Informatics and Mechanics, University of Warsaw, Banacha 2, 02-097 Warsaw, Poland
\hfill\break\indent
{\tt  wojciech.gorny@univie.ac.at }
}
	
%
%
	
	
\setcounter{tocdepth}{1}

\date{\today}
	
\begin{abstract}
In this paper, we study the parabolic and elliptic problems related to the anisotropic $p$-Laplacian operator in the case when it has linear growth on some of the coordinates. In order to define properly a notion of weak solutions and prove their existence, we first construct an anisotropic analogue of Anzellotti pairings and prove a weak Gauss-Green formula which relates the newly constructed pairing with a normal trace of a sufficiently regular vector field. 
\end{abstract}
	
\maketitle

	
\section{Introduction}

The study of evolution problems associated to functionals with linear growth requires quite different techniques with respect to the classical case of quadratic or power-type growth. This is mostly due to the fact that the space $BV(\Omega)$ of functions of bounded variation on a bounded Lipschitz domain $\Omega$ in $\R^N$ (i.e. functions in $L^1(\Omega)$ whose distributional gradient is a finite Radon measure), which is the natural function space for such problems, is neither separable nor reflexive. The most classical example is the total variation flow (see \cite{ACMBook}), which is the gradient flow of the functional 
\begin{equation*}
\mathcal{TV}(u) = \twopartdef{\int_\Omega |Du|}{u \in BV(\Omega) \cap L^2(\Omega);}{+\infty}{u \in L^2(\Omega) \setminus BV(\Omega).}
\end{equation*}
in the Hilbert space $L^2(\Omega)$. Here, $Du$ denotes the distributional gradient of $u$ and $|Du|$ is its total variation. On a formal level, we can associate to it the parabolic equation
\begin{equation}\label{eq:tvflow}
u_t = \Delta_1 u := \mathrm{div}\bigg( \frac{Du}{|Du|} \bigg)
\end{equation}
coupled with appropriate initial and boundary conditions. The operator $\Delta_1$ is called the $1$-Laplacian. The total variation flow was first considered in relation to image processing; the first algorithm for total variation denoising, which is a time-discretisation of the total variation flow, was introduced in 1992 by Rudin, Osher and Fatemi (see \cite{ROF}).

The main difficulty related to the $1$-Laplacian operator is that, unlike the $p$-Laplacian for $p > 1$, i.e.
\begin{equation}
\Delta_p(u) := \mathrm{div}(|\nabla u|^{p-2} \nabla u),
\end{equation}
even for $u \in W^{1,1}(\Omega)$ the $1$-Laplacian is singular on the degenerate points (where $\nabla u = 0$). Therefore, it is not immediately clear how to define the object $\frac{Du}{|Du|}$ over the set $\{ D u = 0 \}$, which might have positive Lebesgue measure. For this reason, one needs to either work directly with the variational formulation, using the notion of variational solutions (see for instance \cite{BDM1}), or find a suitable replacement for $\frac{Du}{|Du|}$; this is typically done using Anzellotti pairings introduced in \cite{Anzellotti} (see also \cite{KT0,KT}).

The basic idea behind the Anzellotti theory can be shortly described as follows. Given a vector field $\z \in L^\infty(\Omega; \mathbb{R}^N)$ and a Sobolev function $u \in W^{1,1}(\Omega)$, the pointwise product $\mathbf{z} \cdot \nabla u$ lies in $L^1(\Omega)$. It appears in the Gauss-Green formula, which is a crucial tool in applications to partial differential equations with linear growth. However, it is not well-defined for general $\z \in L^\infty(\Omega; \mathbb{R}^N)$ and $u \in BV(\Omega)$; the Anzellotti theory consists of introducing a pairing $(\z, Du)$ which is a Radon measure that generalises the pointwise product $\z \cdot \nabla u$ (under a proper joint regularity condition for $\z$ and $u$), and studying the properties of this pairing such as the validity of a co-area formula or a weak Gauss-Green formula. We shortly present the construction from \cite{Anzellotti} in Section \ref{sec:anzellotti}; for some of the more recent developments, see \cite{CF,CCT,CdC}.

Coming back to the total variation flow, the quotient $\frac{Du}{|Du|}$ appearing in equation \eqref{eq:tvflow} is replaced by a vector field $\mathbf{z} \in L^\infty(\Omega; \mathbb{R}^N)$ with divergence in $L^2(\Omega)$ such that $\| \mathbf{z} \|_\infty \leq 1$ and $(\z, Du) = |Du|$. In this sense, it was shown in \cite{ACMBook} that existence and uniqueness of solutions for initial data in $L^2(\Omega)$ follows from the classical theory of maximal monotone operators (see for instance \cite{Brezis}); the authors also prove existence and uniqueness of entropy solutions to \eqref{eq:tvflow} for initial data in $L^1(\Omega)$. Some generalisations of this approach can be found in \cite{GM2021-1,Moll}. The use of Anzellotti-type pairings as a notion to define solutions is common in other problems involving functionals of linear growth, such as the least gradient problem \cite{MRL}, the minimal surface equation \cite{SS}, the Cheeger problem \cite{ACC,CFM}, problems involving Hencky plasticity or elastoplasticity \cite{FG,KT}, hyperbolic conservation laws \cite{CF}, relativistic heat equation \cite{ACMM}, the eigenvalue problem for the $1$-Laplacian \cite{BCN}, the elliptic problem for the $1$-Laplacian with homogeneous Dirichlet boundary condition \cite{Dem,MST}, or elliptic and parabolic problems related to linear growth functionals in a general form \cite{BS,GM2022}.

The first of the main goals of the current paper is to extend the theory due to Anzellotti to a setting involving anisotropic growth. Our main motivation comes from the study of the anisotropic $p$-Laplacian operator, i.e.
\begin{equation}
\Delta_{(p_i)}(u) = \sum_{i = 1}^k \mathrm{div}_{x_i} (|\nabla_{x_i} u|^{p_i - 2} \nabla_{x_i} u)
\end{equation}
which has been typically studied with $1 < p_1 < ... < p_k$. Note that the coordinates corresponding to the same exponent are grouped together - if we simply choose $p_i$ to coincide for some $i$ this leads to a slightly different ``orthotropic'' operator (see \cite{FVV} for a discussion). The anisotropic $p$-Laplacian operator was first studied in relation to regularity of solutions and the required assumptions on the exponents. The first general results concerning this operator and others with similar growth conditions are due to Marcellini \cite{Mar1989,Mar1991}; for further developments, see for instance \cite{AF,BGM,BMS,FS}. A first counterexample showing that some assumption on the closeness of exponents is required for regularity of solutions was already given in \cite{Gia1987}. Some of the more recent results can be found in \cite{AdBF}, \cite{AC}, \cite{FVV} and \cite{FGK}. In particular, the paper \cite{FVV} establishes the main results concerning the anisotropic $p$-Laplacian evolution equation, and one of the main motivations of this paper is to develop a similar theory in the case when we have $p_i = 1$ on some of the coordinates.

When all of the exponents are greater than one, establishing existence of weak solutions to elliptic or parabolic problems involving the operator $\Delta_{(p_i)}$ is relatively straightforward and follows from the reflexivity of the respective function spaces, see \cite{BGM,FVV} or \cite{AC} for the more general case. This not the case once we allow for $p_i = 1$ on some on the coordinates. To the best of the author's knowledge, under this assumption there are currently no results in the literature concerning the parabolic problem, and for the elliptic problem the only results (see \cite{MRST}) concern the equation
\begin{equation}\label{eq:introductionellipticproblem}
-\mathrm{div}_x \bigg( \frac{D_x u}{|D_x u|} \bigg) - \mathrm{div}_y (|\nabla_y u|^{q-2} \nabla_y u) = f(x,y)
\end{equation}
governed by the $1$-Laplacian along some directions (denoted by $x$) and the $q$-Laplacian with $q > 1$ along some other directions (denoted by $y$). It is coupled with homogeneous boundary conditions. One can view the operator on the left hand side as a variant of the anisotropic $p$-Laplacian operator, in which $p_i$ only takes values $1$ and $q$. This equation corresponds to minimisation of the functional
\begin{equation}
J(u) = \int_{\Omega} |D_x u| + \frac1q \int_\Omega |\nabla_y u|^q - \int_\Omega fu
\end{equation}
plus a boundary term corresponding to the homogeneous Dirichlet boundary conditions. Then, the first term has linear growth, and establishing existence of weak solutions to problem \eqref{eq:introductionellipticproblem} requires an anisotropic version of Anzellotti pairings. The construction of such pairings was sketched in \cite{MRST} in the particular case corresponding to equation \eqref{eq:introductionellipticproblem}. The presence of uniformly elliptic terms in some directions significantly changes the problem, and in particular there is no smallness assumptions required on $f$ in order to obtain existence of solutions as for the standard $1$-Laplace equation. In light of the above discussion, our second main goal in this paper is to use the newly constructed strongly anisotropic Anzellotti pairings to prove existence of weak solutions to the parabolic and elliptic problems related to the anisotropic $p$-Laplacian operator with $p_i = 1$ on some of the coordinates.

The paper is organised as follows. In Section \ref{sec:preliminaries}, we recall the construction of standard Anzellotti pairings and recall some classical results concerning gradient flows of convex functionals in Hilbert spaces and convex duality in a setting adapted to calculus of variations. The main part of this paper is contained in Section \ref{sec:anisotropicspaces}, in which we introduce the strongly anisotropic Anzellotti pairings. The basic idea behind the construction is similar to the one in \cite{MRST}, but the definition we give is more general and we fill in the gaps in the sketch of proof given in \cite{MRST}. In particular, we give a correct proof of an anisotropic Gauss-Green formula; for more details we refer to the comment before Proposition \ref{prop:gagliardo}. We also prove a weak variant of the co-area formula for the strongly anisotropic Anzellotti pairing. Then, in Section \ref{sec:dirichlet} we apply the newly constructed pairing to introduce a notion of weak solutions to the anisotropic $p$-Laplace evolution equation 
\begin{equation}
\left\{ \begin{array}{lll} u_t (t,x) = \Delta_{(p_i)} u   \quad &\hbox{in} \ \ (0, T) \times \Omega; \\[5pt] u(t) = 0 \quad &\hbox{on} \ \ (0, T) \times  \partial\Omega; \\[5pt] u(0,x) = u_0(x) \quad & \hbox{in} \ \  \Omega, \end{array} \right.
\end{equation}
where $u_0 \in L^2(\Omega)$. Here, $\Delta_{(p_i)}$ denotes the anisotropic $p$-Laplacian operator, i.e.
\begin{equation}
\Delta_{(p_i)}(u) = \mathrm{div}_{x_1} \bigg( \frac{D_{x_1} u}{|D_{x_1} u|} \bigg) + \sum_{i = 2}^k \mathrm{div}_{x_i} (|\nabla_{x_i} u|^{p_i - 2} \nabla_{x_i} u)
\end{equation}
and prove existence of weak solutions; since the proof is already quite involved, we postpone the discussion on the qualitative properties of solutions to a separate paper. Finally, in Section \ref{sec:elliptic} we consider the elliptic equation
\begin{equation}
\left\{ \begin{array}{lll} -\Delta_{(p_i)} u = f   \quad &\hbox{in} \ \ \Omega; \\[5pt] u = 0 \quad &\hbox{on} \ \ \partial\Omega.
\end{array} \right.
\end{equation}
with $f \in L^{p'_k}(\Omega)$. We prove existence of weak solutions to this problem; in other words, we generalise the results of \cite{MRST} to the full anisotropic $p$-Laplace equation. The proof is based on an application of convex duality as in Section \ref{sec:dirichlet} and does not rely on approximations by more regular anisotropic problems as in \cite{MRST}. Since the argument is very similar to the one in the previous section, we present only a simplified version of the proof, and for a discussion on regularity of solutions we refer to \cite{MRST}.

\section{Preliminaries}\label{sec:preliminaries}

\subsection{BV functions and Anzellotti pairings}\label{sec:anzellotti}

In the minimisation of functionals with linear growth, the natural energy space to study the problem is the space of functions of bounded variation. As a preparation for the anisotropic case, in which we have linear growth only on some of the coordinates, let us recall several facts concerning functions of bounded variation (we follow the presentation in \cite{AFP}). We also recall the basic idea behind the Anzellotti construction of the weak normal traces and a weak Gauss-Green formula \cite{Anzellotti}; we prove a generalised version of these results in Section \ref{sec:anisotropicspaces}. Throughout the whole subsection, we assume that $\Omega \subset \R^N$ is an open bounded set with Lipschitz boundary.

\begin{definition}{\rm
We say that $u \in L^1(\Omega)$ is a function of bounded variation if its gradient $Du$ in the sense of distributions is a finite (vectorial) Radon measure, i.e.
\begin{equation}
\int_{\Omega} u \, \mathrm{div}(\varphi) \, dx = -\int_{\Omega} \varphi \, d Du
\end{equation}
for all $\varphi \in C_c^{\infty}(\Omega; \mathbb{R}^N)$. The space of such functions will be denoted by $BV(\Omega)$. In other words, we have
\begin{equation}
BV(\Omega) = \bigg\{ u \in L^1(\Omega): \,\, Du \in \mathcal{M}(\Omega; \R^{N}) \bigg\}.
\end{equation}
The distributional gradient $Du$ is a vector-valued measure with finite total variation
\begin{equation}
\int_\Omega |Du| = \sup \bigg\{ \int_{\Omega} \, u \, \mathrm{div} (\varphi) \, dx: \,\, \varphi \in C_c^{\infty}(\Omega; \R^N), \,\, |\varphi(x)| \leq 1
\, \, \hbox{for $x \in \Omega$} \bigg\}.
\end{equation}
The space $BV(\Omega)$ is endowed with the norm
\begin{equation}
\| u \|_{BV(\Omega)} = \| u \|_{L^1(\Omega)} +
\vert  Du \vert (\Omega).
\end{equation}
}
\end{definition}

Quite often, convergence in norm in $BV(\Omega)$ is a too strong requirement; instead, one may consider weak* convergence, i.e.
\begin{equation}
u_n \rightharpoonup u \hbox{ weakly* in } BV(\Omega)  \Leftrightarrow   u_n \rightarrow u \hbox{ in } L^1(\Omega) \hbox{ and } \sup_n \int_\Omega |Du_n| < \infty;
\end{equation}
or strict convergence, i.e.
\begin{equation}
u_n \rightarrow u \hbox{ strictly in } BV(\Omega) \Leftrightarrow   u_n \rightarrow u \hbox{ in } L^1(\Omega) \hbox{ and } \int_\Omega |Du_n| \rightarrow \int_\Omega |Du|.
\end{equation}
Then, one can prove that for every $u \in BV(\Omega)$ there exists a sequence of smooth functions which converges strictly to $u$, see for instance \cite{AFP}; unlike the Sobolev spaces, this result does not hold true if we consider norm convergence in place of strict convergence.

An important property of BV functions, which follows directly from the above characterisation of $\int_\Omega |Du|$, is the following co-area formula. From now on, for a measurable function $u: \Omega \rightarrow \R$, we denote its superlevel sets by $E_t = \{ u \geq t \}$.

\begin{theorem}
Let $u \in BV(\Omega)$. Then, for almost all $t \in \R$ the set $E_t$ has finite perimeter (i.e. $\chi_{E_t} \in BV(\Omega)$) and
\begin{equation}\label{eq:isotropiccoarea}
\int_\Omega |Du| = \int_{-\infty}^\infty \int_\Omega |D\chi_{E_t}| \, dt.
\end{equation}
Furthermore, whenever $u \in L^1(\Omega)$ is such that the expression on the right-hand side of \eqref{eq:isotropiccoarea} is finite, we have that $u \in BV(\Omega)$.
\end{theorem}

Moreover, we have the following trace theorem for BV functions.

\begin{theorem}\label{thm:standardtracetheorem}
There exists a linear, continuous, and surjective operator $T: BV(\Omega) \rightarrow L^1(\partial\Omega)$ such that
\begin{equation}
\int_\Omega u \, \mathrm{div} (\varphi) \, dx + \int_\Omega \varphi \, dDu = \int_{\partial\Omega} u \, \varphi \cdot \nu^\Omega \, d\mathcal{H}^{N-1}
\end{equation}
for all $u \in BV(\Omega)$ and $\varphi \in C^1(\overline{\Omega}; \mathbb{R}^N)$. Here, we denote the trace $T(u)$ of a BV function on $\partial\Omega$ simply by $u$ (it agrees with $u$ for continuous functions), and $\nu^\Omega$ denotes the outer unit normal to $\Omega$.
\end{theorem}

Now, for $u \in BV(\Omega)$ denote by $\frac{Du}{|Du|}$ the Radon-Nikodym derivative of the measure $Du$ with respect to the measure $\vert Du \vert$. As a consequence of the co-area formula, we have that for almost all $t \in \R$ 
$$ \frac{D \chi_{E_t}}{|D \chi_{E_t}|} = \frac{Du}{|Du|} \quad |D\chi_{E_t}|-\mbox{a.e. in } \Omega.$$
We now recall the definition of classical Anzellotti pairings. 
Denote by $p'$ the dual exponent to $p$. Following \cite{Anzellotti}, we introduce the space
\begin{equation}
\label{xw}
X_p(\Omega) = \{ \z \in L^{\infty}(\Omega; \R^N):
\divi(\z)\in L^p(\Omega) \}.
\end{equation}

\begin{definition}{\rm
For $\z \in X_p(\Omega)$ and $u \in BV(\Omega) \cap L^{p'}(\Omega)$,
define the functional $(\z,Du): C_c^{\infty}(\Omega) \rightarrow \R$
by the formula
\begin{equation}\label{Anzel1}
\langle (\z,Du),\varphi \rangle = - \int_{\Omega} u \, \varphi \, \divi(\z) \, dx -
\int_{\Omega} u \, \z \cdot \nabla \varphi \, dx.
\end{equation}
}

\end{definition}

The following result collects some of the most important properties of the pairing $(\z, Du)$, formally defined only as a distribution on $\Omega$.

\begin{proposition}
The distribution $(\z,Du)$ is a Radon measure in $\Omega$. Moreover,
\begin{equation}
\bigg\vert \int_{B} (\z,Du) \bigg\vert \leq \int_{B} |(\z,Du)| \leq \| \z
\|_{\infty}
\int_{B} \vert Du \vert
\end{equation}
for any Borel set $B \subseteq \Omega$. In particular, $(\z,Du)$ is absolutely continuous with respect to $\vert Du \vert$. Furthermore, for all $w \in W^{1,1}(\Omega)$ we have
\begin{equation}
\int_{\Omega} (\z,Dw) = \int_{\Omega} \z \cdot \nabla w \, dx,
\end{equation}
so $(\z,Du)$ agrees on Sobolev functions with the dot product of $\z$ and $\nabla u$.
\end{proposition}

In \cite{Anzellotti}, a weak trace on $\partial \Omega$ of the normal component of $\z \in X_p(\Omega)$ is defined. To be exact, it is shown that there exists a linear operator $\gamma : X_1(\Omega) \rightarrow L^{\infty}(\partial \Omega)$ such that $$\Vert \gamma(\z) \Vert_{\infty} \leq \Vert \z \Vert_{\infty}$$
and
$$\gamma(\z) (x) = \z(x) \cdot \nu^\Omega(x) \ \ \ \ {\rm for \ all} \ x \in
\partial
\Omega \ \ {\rm if} \
\ \z \in C^1(\overline{\Omega}; \R^N),$$
where $\nu^\Omega$ is the outer unit normal to $\partial\Omega$. We shall denote \ $\gamma (\z)(x)$ by $[\z, \nu^\Omega](x)$. Moreover, we recall the following two properties of the pairing proved in the same paper: the Gauss-Green formula, which relates the function $[\z, \nu^\Omega]$ and the measure $(\z, Dw)$, and a co-area formula for $(\z, Du)$.

\begin{theorem}
For all $\z \in X_p(\Omega)$ and $u \in BV(\Omega) \cap L^{p'}(\Omega)$, we have
\begin{equation}\label{thm:isotropicgreenformula}
\int_{\Omega} u \, \mathrm{div} (\z) \, dx + \int_{\Omega} (\z, Du) = \int_{\partial \Omega} u \, [\z, \nu^\Omega] \, d\mathcal{H}^{N-1}.
\end{equation}
\end{theorem}

\begin{theorem}\label{thm:isotropiccoarea}
For all $\z \in X_p(\Omega)$ and $u \in BV(\Omega) \cap L^{p'}(\Omega)$, the equality
\begin{equation*}
\langle (\z, Du), \varphi \rangle = \int_{-\infty}^\infty \langle (\z, D\chi_{E_t}),\varphi \rangle \, dt
\end{equation*}
holds for all functions $\varphi \in C_c^\infty(\Omega)$. Moreover,
\begin{equation*}
\int_\Omega (\z, Du) = \int_{-\infty}^\infty \int_\Omega (\z, D\chi_{E_t}) \, dt.
\end{equation*}
\end{theorem}

\subsection{Gradient flows in Hilbert spaces}\label{sec:breziskomura}

We now introduce the basic notation concerning multivalued operators in Hilbert spaces and recall some classical results on semigroup solutions to gradient flows. We follow the presentation in \cite{Brezis}. 

Let $H$ be a real Hilbert space. Given a functional $\mathcal{F}: H \to (-\infty, \infty]$, we call the set
$$D(\mathcal{F}) : = \{ u \in H: \ \mathcal{F}(u) < + \infty \}$$
the effective domain of $\mathcal{F}$. The functional $\mathcal{F}$ is said to be proper if $D(\mathcal{F})$ is nonempty.
Furthermore, we say that $\mathcal{F}$ is lower semicontinuous if for every $c \in \R$ the sublevel set $\{ u \in D(\mathcal{F}): \ \mathcal{F}(u) \leq c \}$ is closed in $H$.

Given a proper and convex functional $\mathcal{F} : H \to (-\infty, \infty]$, its {\it subdifferential} is the set
\begin{equation}\label{subdif2}\partial \mathcal{F} := \left\{(u,h) \in H \times H: \ \mathcal{F}(u+v) - \mathcal{F}(u) \geq \langle h, v \rangle_H \ \ \forall \, v \in H \right\}.
\end{equation}
It is a generalisation of the derivative; in the case when $\mathcal{F}$ is Fr\'{e}chet differentiable, its subdifferential is single-valued and equals the Fr\'{e}chet derivative.

Now, consider a multivalued operator $A$ on $H$, i.e. a mapping $A: H \rightarrow 2^H$. It is standard to identify $A$ with its graph in the following way: for every $u \in H$, we set
$$ Au : = \left\{ v \in H: \  (u,v) \in A \right\}.$$
We denote the domain of $A$ by
$$D (A) : = \left\{ u \in H: \ Au \not= \emptyset \right\}$$
and the range of $A$ by
$$R(A) : = \bigcup_{u \in D(A)} Au.$$
A multivalued operator $A$ on $H$ is {\it monotone} if
$$\langle u - \hat{u}, v - \hat{v} \rangle_H \geq 0$$
for all $(u,v), (\hat{u}, \hat{v}) \in A.$
If there is no monotone operator which strictly contains $A$, we say that $A$ is maximal monotone. A classical example of a multivalued operator is the subdifferential; if $\mathcal{F} : H \to (-\infty, \infty]$ is convex and lower semicontinuous, then $\partial \mathcal{F}$ is a maximal monotone multivalued operator on $H$.

For $1 \leq p < \infty$, we denote
$$L^p(a,b; H):= \left\{ u : [a,b] \rightarrow H \ \hbox{measurable such that} \  \int_a^b \Vert u(t) \Vert_H^p \, dt < \infty \right\}$$
and
\begin{align}
W^{1,p}(a,b;H):= &\bigg \{ u \in L^p(a,b; H) \ \hbox{and} \ \exists \, v \in L^p(a,b; H): \\
&\qquad\qquad\qquad\qquad u(t) - u(a) = \int_a^t v(s) \, ds \ \ \forall \, t \in (a,b)  \bigg \}.
\end{align}
By \cite[Corollaire A.2]{Brezis}, if $u \in W^{1,p}(a,b;H)$, it is differentiable for almost all $t \in (a,b)$ and
$$u(t) - u(a) = \int_a^t \frac{du}{dt}(s) \, ds \ \ \forall \, t \in (a,b).$$
We also set $W_{\rm loc}^{1,p}(0,T;H)$ to be the space of all functions $u$ with the following property: for all $0 < a < b < T$, we have that $u \in W^{1,p}(a,b;H)$.

Consider the abstract Cauchy problem
\begin{equation}\label{eq:abstractcauchy}
\left\{ \begin{array}{ll} 0 \in \frac{du}{dt} + \partial \mathcal{F} (u(t)) \, \quad &\mbox{for } t \in (0, T);  \\
[10pt] u(0) = u_0, \quad & \end{array} \right.
\end{equation}
where $u_0 \in H$.

\begin{definition} 
We say that $u \in C([0,T]; H)$ is a strong solution to problem \eqref{eq:abstractcauchy}, if the following conditions hold: $u \in W_{\rm loc}^{1,2}(0,T;H)$; for almost all $t \in (0,T)$ we have $u(t) \in D(\partial \mathcal{F})$; and it satisfies \eqref{eq:abstractcauchy}.
\end{definition}

The following result, called the Brezis-Komura theorem, summarises the existence theory for solutions to the abstract Cauchy problem \eqref{eq:abstractcauchy} obtained using the semigroup approach (see \cite{Brezis}). We refer to \cite{Brezis} for an extensive discussion of additional properties of solutions, such as regularity of time derivative and the $T$-contraction property.

\begin{theorem}\label{thm:breziskomura}
Let $\mathcal{F} : H \to (-\infty, \infty]$ be a proper, convex, and lower semicontinuous functional. Given $u_0 \in \overline{D(\mathcal{F})}$, there exists a unique strong solution to the abstract Cauchy problem \eqref{eq:abstractcauchy}. If additionally $u_0 \in D(\mathcal{F})$, then $u \in W^{1,2}(0, T; H)$. 
\end{theorem}

In this paper, we will only apply this result with $H = L^2(\Omega)$; this will be in the part concerning the anisotropic $p$-Laplacian evolution equation in Section \ref{sec:dirichlet}.

We conclude this subsection by the following result given in \cite[Proposition 2.11]{Brezis}, which is useful to characterise the closure of the domain of the subdifferential.

\begin{proposition}\label{prop:domain} 
Let $\mathcal{F}: H \to (-\infty, \infty]$ be a proper, convex, and lower semicontinuous functional. Then,
$$D(\partial \mathcal{F}) \subset D( \mathcal{F}) \subset \overline{ D( \mathcal{F})} \subset \overline{ D( \partial \mathcal{F})}.$$
\end{proposition}

\subsection{Convex duality in calculus of variations}\label{subsec:convexduality}

One of the classical tools used to characterise the subdifferential of a convex functional $\mathcal{F}$ is convex duality in a setting typical for calculus of variations (a standard reference is \cite[Chapter III.4]{EkelandTemam}). First, let us recall the notion of the Legendre-Fenchel transform. It is defined as follows: given a Banach space $V$ and $F: V \rightarrow \mathbb{R} \cup \{ + \infty \}$, we define $F^*: V^* \rightarrow \mathbb{R} \cup \{ + \infty \}$ by the formula
\begin{equation*}
F^*(v^*) = \sup_{v \in V} \bigg\{ \langle v, v^* \rangle - F(v) \bigg\}.
\end{equation*}
Then, let us recall shortly how the dual problem is typically defined in the setting of calculus of variations. 

Let $U, V$ be two Banach spaces and let $A: U \rightarrow V$ be a continuous linear operator. Denote by $A^*: V^* \rightarrow U^*$ its dual. Then, if the primal problem is of the form
\begin{equation}\tag{P}\label{eq:primal}
\inf_{u \in U} \bigg\{ E(Au) + G(u) \bigg\},
\end{equation}
where $E: V \rightarrow (-\infty,+\infty]$ and $G: U \rightarrow (-\infty,+\infty]$ are proper, convex and lower semicontinuous, then the dual problem is defined as the maximisation problem
\begin{equation}\tag{P*}\label{eq:dual}
\sup_{v^* \in V^*} \bigg\{ - E^*(-v^*) - G^*(A^* v^*) \bigg\}
\end{equation}
It turns out that in the above setting the minimum value in the primal problem equals the maximum value in the dual problem; this is the content of the Fenchel-Rockafellar duality theorem, which we recall in the form given in \cite[Remark III.4.2]{EkelandTemam}. 

\begin{theorem}[Fenchel-Rockafellar duality theorem]\label{thm:fenchelrockafellar}
Suppose that $E$ and $G$ are proper, convex and lower semicontinuous. If there exists $u_0 \in U$ such that $E(A u_0) < \infty$, $G(u_0) < \infty$ and $E$ is continuous at $A u_0$, then
$$\inf \eqref{eq:primal} = \sup \eqref{eq:dual}$$
and the dual problem \eqref{eq:dual} has at least one solution. Moreover, if $u$ is a solution of \eqref{eq:primal} and $v^*$ is a solution of \eqref{eq:dual}, the following extremality conditions hold:
\begin{equation}
E(Au) + E^*(-v^*) = \langle -v^*, Au \rangle;
\end{equation}
\begin{equation}
G(u) + G^*(A^* v^*) = \langle u, A^* v^* \rangle.
\end{equation}
\end{theorem}

In the case when the primal problem does not have any solutions, we may use the $\varepsilon-$subdifferentiability property of minimising sequences, see \cite[Proposition V.1.2]{EkelandTemam}: for any minimising sequence $u_n$ for \eqref{eq:primal} and a maximiser $v^*$ of \eqref{eq:dual}, we have
\begin{equation}\label{eq:epsilonsubdiff1}
0 \leq E(Au_n) + E^*(-v^*) - \langle -v^*, Au_n \rangle \leq \varepsilon_n
\end{equation}
and
\begin{equation}\label{eq:epsilonsubdiff2}
0 \leq G(u_n) + G^*(A^* v^*) - \langle u_n, A^* v^* \rangle \leq \varepsilon_n
\end{equation}
with $\varepsilon_n \rightarrow 0$.

\section{Strongly anisotropic Anzellotti pairings}\label{sec:anisotropicspaces}

In this Section, we construct an anisotropic counterpart of the Anzellotti pairings described in Section \ref{sec:anzellotti}. We begin by introducing the notation concerning the coordinates and the domain which will be valid throughout the rest of the paper.

\subsection{Notation}

Let $\Omega_1 \subset \R^{n_1}$ and $\Omega_2 \subset \R^{n_2+...+n_k}$ be bounded open sets. We denote by $x_1$ the points in $\Omega_1$ and by $(x_2,.,,,x_k)$ the points in $\Omega_2$; note that the points $x_i$ have $n_i$ coordinates. We set
$$\Omega = \Omega_1 \times \Omega_2 \subset \R^{N},$$
where $N = n_1 + ... + n_k$, and assume that $\Omega$ has Lipschitz boundary. For a sufficiently regular function $u: \Omega \rightarrow \R$, we denote by $\nabla_{x_i}$ the vector consisting of derivatives in directions $x_i$. Therefore, the full gradient of $u$ is 
$$\nabla u = (\nabla_{x_1} u, ..., \nabla_{x_k} u).$$
We use a similar notation for the divergence; for a sufficiently regular vector field $\varphi: \Omega \rightarrow \R^N$, we denote by $\mathrm{div}_{x_i}$ the sum of derivatives in directions $x_i$. We also use this notation for vector fields with values in $\R^{n_i}$. Therefore, the full divergence of $\varphi$ is
$$ \mathrm{div}(\varphi) = \sum_{i = 1}^k \mathrm{div}_{x_i}(\varphi). $$
One of the main technical difficulties we face in this paper is that the estimates for the full divergence $\mathrm{div}(\varphi)$ of a vector field $\varphi$ do not imply estimates for the restricted divergences $\mathrm{div}_{x_i}(\varphi)$; as a consequence, we usually cannot split our reasoning into coordinates.

Finally, to each direction $x_i$ we associate an exponent $p_i$; throughout the paper, we assume that $p_1 = 1$, and without loss of generality we assume that $1 < p_2 \leq ... \leq p_k$. The case when we have more than one subspace corresponding to $p = 1$ follows with a minor modification of the proofs.

\subsection{Suitable anisotropic spaces}

Our first goal is to identify the proper anisotropic function spaces for the study of functionals of the type
\begin{equation}
\mathcal{F}(u) = \sum_{i = 1}^k \frac{1}{p_i} \int_\Omega |\nabla_{x_i} u|^{p_i} \, dx.
\end{equation}
We will work with closed subspaces of $BV(\Omega)$, so we can assume a priori that the functions we consider have partial derivatives in the sense of distributions which are finite Radon measures. We define the Sobolev space $W^{1,(p_i)}(\Omega)$ in the following way: 
\begin{equation}
W^{1,(p_i)}(\Omega) = \bigg\{ u \in W^{1,1}(\Omega): \quad \nabla_{x_i} u \in L^{p_i}(\Omega; \R^{n_i}) \mbox{ for all } i = 1,...,k \bigg\},
\end{equation}
and it is equipped with the norm
\begin{equation}
\| u \|_{W^{1,(p_i)}(\Omega)} = \| u \|_{L^1(\Omega)} + \sum_{i=1}^k \| \nabla_{x_i} u \|_{L^{p_i}(\Omega; \R^{n_i})}.
\end{equation}
Similarly to the classical Sobolev spaces, we can approximate in the norm every function in $W^{1,(p_i)}(\Omega)$ using smooth functions; we postpone the proof to the next Section.

We will additionally consider the Sobolev spaces with zero trace in the following two variants. The first one is a classical zero-trace space
\begin{equation}
W_0^{1,(p_i)}(\Omega) = \bigg\{ u \in W^{1,1}_0(\Omega): \quad \nabla_{x_i} u \in L^{p_i}(\Omega; \R^{n_i}) \mbox{ for all } i = 1,...,k \bigg\}.
\end{equation}
However, in order to study the functionals of the form $\mathcal{F}$, this requirement is a bit too strong; therefore, we introduce the following {\it weak zero-trace} space, in which the trace condition is imposed on all coordinates except for $x_1$:
\begin{align}
W_{0,w}^{1,(p_i)}(\Omega) = \bigg\{ u \in W^{1,(p_i)}(\Omega): \,\, \mbox{ the function } &(x_2,...,x_k) \mapsto u(x_1,...,x_k) \\
&\mbox{belongs to } W_0^{1,(p_i)_{i=2}^k}(\Omega_2) \mbox{ for a.e. } x_1 \in \Omega_1 \bigg\}.
\end{align}
In a similar fashion, we introduce an anisotropic space of functions of bounded variation $BV^{(p_i)}(\Omega)$, in which the distributional derivative in the first coordinate is a finite Radon measure and the other distributional derivatives are functions; to be exact, we set
\begin{equation}
BV^{(p_i)}(\Omega) = \bigg\{ u \in BV(\Omega): \,\, D_{x_1} u \in \mathcal{M}(\Omega; \R^{n_1}), \,\, \nabla_{x_i} u \in L^{p_i}(\Omega; \R^{n_i}) \mbox{ for all } i = 2,...,k \bigg\},
\end{equation}
and the norm is given by
\begin{equation}
\| u \|_{BV^{(p_i)}(\Omega)} = \| u \|_{L^1(\Omega)} + \int_\Omega |D_{x_1} u| + \sum_{i=2}^k \| \nabla_{x_i} u \|_{L^{p_i}(\Omega; \R^{n_i})}.
\end{equation}
Similarly to the Sobolev case, we define the zero-trace space as
\begin{equation}
BV_0^{(p_i)}(\Omega) = \bigg\{ u \in BV_0(\Omega): \,\, D_{x_1} u \in \mathcal{M}(\Omega; \R^{n_1}), \,\, \nabla_{x_i} u \in L^{p_i}(\Omega; \R^{n_i}) \mbox{ for all } i = 2,...,k \bigg\},
\end{equation}
and the weak zero-trace space as
\begin{align}
BV_{0,w}^{(p_i)}(\Omega) = \bigg\{ u \in BV^{(p_i)}(\Omega): \,\, \mbox{ the function } &(x_2,...,x_k) \mapsto u(x_1,...,x_k) \\
&\mbox{belongs to } W_0^{1,(p_i)_{i=2}^k}(\Omega_2) \mbox{ for a.e. } x_1 \in \Omega_1 \bigg\}.
\end{align}
For the choice $(p_1,p_2) = (1,q)$, the space $BV_{0,w}^{(p_i)}(\Omega)$ corresponds to the space $BV^{(q)}$ used in \cite{MRST}.

Observe that all of the above spaces are subspaces of $BV(\Omega)$; they are also subspaces of an even larger space, i.e.
\begin{equation}
BV_{x_1}(\Omega) = \bigg\{ u \in L^1(\Omega): \,\, D_{x_1} u \in \mathcal{M}(\Omega; \R^{n_1}) \bigg\}.
\end{equation}
Similarly to the standard BV space, we can show (e.g. following the argument in \cite{AFP}) that
\begin{equation}\label{dfn:tvonx1}
\int_\Omega |D_{x_1} u| = \sup \bigg\{ \int_{\Omega} u \, \mathrm{div}_{x_1} (\varphi) \, dx: \,\, \varphi \in C^{\infty}_{0}(\Omega; \R^{n_1}), \,\, |\varphi(x)| \leq 1
\, \, \hbox{for $x \in \Omega$} \bigg\},
\end{equation}
and as a consequence the total variation $|D_{x_1} u|$ is lower semicontinuous with respect to convergence in $L^1(\Omega)$. Similarly, by considering the zero extension of $u$ to a larger domain $\Omega' = \Omega'_1 \times \Omega_2$, we can see that the functional
\begin{equation*}
u \mapsto \int_\Omega |D_{x_1} u| + \int_{\partial\Omega_1 \times \Omega_2} |u| \, d\mathcal{H}^{N-1}
\end{equation*}
is lower semicontinuous with respect to convergence in $L^1(\Omega)$.

Since all functions $u \in BV_{0,w}^{(p_i)}(\Omega)$ have zero trace in the coordinates $x_2, ..., x_k$ and the standard proof of the Poincar\'e inequality involves only an estimate in one direction, we get that $u \in L^{p_k}(\Omega)$ and the $L^{p_k}$-norm can be estimated using only the gradient of $u$ restricted to the direction $x_k$. To be exact, we have the following result.

\begin{proposition}[Poincar\'e inequality]\label{prop:poincare}
For every $u \in BV_{0,w}^{(p_i)}(\Omega)$ we have
\begin{equation}
\int_\Omega |u|^{p_k} \, dx \leq C(\Omega) \int_\Omega |\nabla_{x_k} u|^{p_k} \, dx,
\end{equation}
so in particular $BV_{0,w}^{(p_i)}(\Omega) \hookrightarrow L^{p_k}(\Omega)$.
\end{proposition}

A sketch of proof in an anisotropic case involving only two exponents was given in \cite[Theorem 2.3]{MRST}. As a consequence, the norms
\begin{equation}
\| u \|_{W_{0,w}^{1,(p_i)}(\Omega)} = \sum_{i=1}^k \| \nabla_{x_i} u \|_{L^{p_i}(\Omega; \R^{n_i})}
\end{equation}
and
\begin{equation}
\| u \|_{BV_{0,w}^{(p_i)}(\Omega)} = \int_\Omega |D_{x_1} u| + \sum_{i=2}^k \| \nabla_{x_i} u \|_{L^{p_i}(\Omega; \R^{n_i})}
\end{equation}
are equivalent to the Sobolev and BV norms respectively on the weak zero-trace space. 

Finally, let us briefly comment on the Sobolev-type embeddings for anisotropic Sobolev functions. Assume that
\begin{equation}\label{eq:conditionforpi}
\sum_{i = 1}^k \frac{n_i}{p_i} > 1
\end{equation}
and denote by
\begin{equation}
\overline{p} = \frac{N}{(\sum_{i = 1}^k \frac{n_i}{p_i}) - 1}
\end{equation}
the harmonic mean of the exponents $p_i$, weighted by the respective dimensions $n_i$ of subspaces corresponding to the exponents $p_i$. A classical estimate given in \cite{Tro} (see also \cite{Ada}) shows that there exists $C > 0$ such that for every $u \in C_c^\infty(\R^N)$ we have
\begin{equation}
\| u \|_{\overline{p}} \leq C \sum_{i = 1}^k \| \nabla_{x_i} u \|_{p_i}.
\end{equation}
Corresponding results for bounded domains without the requirement that the function has zero trace are only valid under some additional assumptions on the structure of the domain; then, they translate to the validity of a Sobolev embedding. A classical reference is \cite{KK}, where it is shown that
\begin{equation}
W^{1,(p_i)}(\Omega) \hookrightarrow L^{\overline{p}}(\Omega)
\end{equation}
under an assumption that covers rectangular domains. A similar result if the opposite inequality in \eqref{eq:conditionforpi} holds gives a Morrey-type embedding into H\"older continuous functions. In this paper, we will not directly use this type of results, so for more recent advances and a discussion on the shape of the domain we just refer the interested reader to \cite{FGK,Rak1,Rak2}.

\subsection{Strongly anisotropic Anzellotti pairings}

We now present an anisotropic version of the Anzellotti construction, which is more suitable for handling the anisotropic $p$-Laplacian operator in the case when we have a linear growth term on some of the coordinates. A first construction of this type appeared in \cite{MRST} in a restricted setting adapted to equation \eqref{eq:introductionellipticproblem}. The basic idea behind parts of the construction presented here is similar to the one in \cite{MRST}; however, our definition is more general and we provide additional details in the proofs. In particular, we fill a small gap in the proof of the Gauss-Green formula in \cite{MRST}; the result itself remains true under slightly stronger assumptions on the domain. We further comment on this issue just before Proposition \ref{prop:gagliardo}.

We start our considerations by proving an anisotropic counterpart of the Meyers-Serrin theorem, i.e. an approximation by smooth functions in an analogue of strict topology for BV functions.

\begin{proposition}\label{prop:meyersserrin}
Take any $q \in [1,\infty]$ and let $u \in BV^{(p_i)}(\Omega) \cap L^q(\Omega)$. Then, there exists a sequence $u_n \in W^{1,1}(\Omega) \cap C^\infty(\Omega)$ such that:
\begin{enumerate}
\item $u_n \rightarrow u$ in $L^q(\Omega)$ for $q < \infty$ and $u_n \rightharpoonup u$ weakly* in $L^\infty(\Omega)$ for $q = \infty$;

\item $\int_\Omega |\nabla_{x_1} u_n| \, dx \rightarrow \int_{\Omega} |D_{x_1} u|$;

\item $\nabla_{x_i} u_n \rightarrow \nabla_{x_i} u$ in $L^{p_i}(\Omega; \mathbb{R}^{n_i})$ for all $i = 2,...,k$;

\item $u_n|_{\partial\Omega_1 \times \Omega_2} = u|_{\partial\Omega_1 \times \Omega_2}$.
\end{enumerate}
If $u \in W^{1,(p_i)}(\Omega)$, we also have that $\nabla_{x_1} u_n \rightarrow \nabla_{x_1} u$ in $L^{1}(\Omega; \mathbb{R}^{n_1})$.
\end{proposition}

\begin{proof}
We proceed similarly to the proof of the corresponding result for BV functions given in \cite[Theorem 3.9]{AFP} or \cite[Theorem 1.17]{Giu}. Take a sequence of open sets $\Omega_j$ with the following property: $\Omega_j \Subset \Omega$ and every point $x \in \Omega$ lies in at most four sets $\Omega_j$. Take a partition of unity $\varphi_j$ relative to this covering, i.e. $\varphi_j \in C_c^\infty(\Omega)$, $\varphi_j \geq 0$, $\mathrm{supp}(\varphi_j) \subset \Omega_j$ and $\sum_{j=1}^\infty \varphi_j \equiv 1$ in $\Omega$.

Let $\rho_\varepsilon$ be a family of standard mollifiers and take $\delta \in (0,1)$. Then, for every $j \in \mathbb{N}$ there exists $\varepsilon_j > 0$ such that $\mathrm{supp}(\rho_{\varepsilon_j} * (\varphi_j u)) \subset \Omega_j$ and the following conditions hold:
\begin{equation}
\int_{\Omega} |\rho_{\varepsilon_j} * (\varphi_j u) - \varphi_j u|^q \, dx < (2^{-j} \delta)^q;
\end{equation}
\begin{equation}
\int_{\Omega} |\rho_{\varepsilon_j} * (u \nabla_{x_1} \varphi_j ) - u \nabla_{x_1} \varphi_j| \, dx < 2^{-j} \delta;
\end{equation}
and for $i = 2, ..., k$ we have
\begin{equation}
\int_{\Omega} |\rho_{\varepsilon_j} * \nabla_{x_i}(\varphi_j u) - \nabla_{x_i}(\varphi_j u)|^{p_i} \, dx < (2^{-j} \delta)^{p_i}.
\end{equation}
Then, we set $u_\delta = \sum_{j = 1}^\infty \rho_{\varepsilon_j} * (u \varphi_j)$. The function $u_\delta$ is smooth because each of the terms is smooth and the sum is locally finite. Our choice of the sequence $\varepsilon_j$ yields that
\begin{equation}
\bigg( \int_\Omega |u_\delta - u|^q \, dx \bigg)^{1/q} \leq \sum_{j = 1}^\infty \bigg( \int_{\Omega} |\rho_{\varepsilon_j} * (\varphi_j u) - \varphi_j u|^q \, dx \bigg)^{1/q} < \delta;
\end{equation}
similarly,
\begin{equation}
\int_\Omega |D_{x_1} u_\delta| = \int_\Omega |\nabla_{x_1} u_\delta| \, dx < \sum_{j=1}^\infty \int_{\Omega} \varphi_j |D_{x_1} u| + \delta = \int_\Omega |D_{x_1} u| + \delta
\end{equation}
and for $i = 2,...,k$ we have
\begin{equation}
\bigg( \int_\Omega |\nabla_{x_i} u_\delta - \nabla_{x_i} u|^{p_i} \, dx \bigg)^{1/p_i} \leq \sum_{j = 1}^\infty \bigg( \int_{\Omega} |\rho_{\varepsilon_j} * \nabla_{x_i}(\varphi_j u) - \nabla_{x_i}(\varphi_j u)|^{p_i} \, dx \bigg)^{1/p_i} < \delta.
\end{equation}
Since $\delta \in (0,1)$ was arbitrary, we conclude the proof of points (a)-(c). Because we assumed that $\partial\Omega$ is Lipschitz, point (d) follows similarly as in \cite{Giu}. Finally, the claim for functions in $W^{1,(p_i)}(\Omega)$ follows in the same way as point (c).
\end{proof}

We can now define the strongly anisotropic Anzellotti pairings. We first introduce a space of admissible vector fields with integrable divergence, i.e.
\begin{equation}
X^{(p'_i)}_r(\Omega) := \bigg\{ \mathbf{z} \in L^{p'_1}(\Omega; \R^{n_1}) \times ... \times L^{p'_k}(\Omega; \R^{n_k}): \, \mathrm{div}(\mathbf{z}) \in L^r(\Omega) \bigg\}.
\end{equation}
Here, $p'_i$ is the dual exponent to $p_i$. We denote the vector fields in $X^{(p'_i)}_r(\Omega)$ in the following way: $\z = (\z_1, ..., \z_k)$, where $\z_i \in L^{p'_i}(\Omega; \R^{n_i})$. Since in this paper we assume that $p_1 = 1$, the first component of every vector field in $X^{(p'_i)}_r(\Omega)$ is bounded.

Similarly to the classical case, we now introduce the following joint condition on the function $u$ and vector field $\mathbf{z}$ which allows us to define Anzellotti pairings. From now on, we assume that
\begin{equation}\label{eq:conditionforuz}
u \in BV^{(p_i)}(\Omega) \cap L^{\max(p_k,q)}(\Omega) \quad \mbox{and} \quad \mathbf{z} \in X^{(p'_i)}_{q'}(\Omega).
\end{equation}
The main setting to which we apply this construction is when $u \in BV_{0,w}^{(p_i)}(\Omega)$; then, we have the Poincar\'e inequality, and as a consequence we have $u \in L^{p_k}(\Omega)$ and it is sufficient to take $q = p_k$. However, it is not the only possible choice, and for instance one can take $\Omega$ to be sufficiently regular so that a Sobolev embedding $BV^{(p_i)}(\Omega) \xhookrightarrow{} L^{\overline{p}}(\Omega)$ holds; then, the assumption on $q$ can be interpreted as a closeness condition for the exponents.

For the standard Anzellotti pairings, usually the case $N = 1$ is considered separately due to the fact that the divergence is just the derivative and vector fields with integrable divergence are Sobolev functions. In our case, when the linear-growth space is one-dimensional, i.e. $n_1 = 1$, it will not make any difference in the proofs, because we will always work with the full divergence $\mathrm{div}(\z)$; since we only assume that \eqref{eq:conditionforuz} holds, we do not know if $\mathrm{div}_{x_i}(\z)$ is integrable for any $i = 1, ..., N$, so we cannot split the divergence into coordinates in the proofs. This will be the main difficulty we face in this Section.

\begin{definition}\label{dfn:anzellottipairing}
Assume that the pair $(u,\z)$ satisfies condition \eqref{eq:conditionforuz}. For all $\varphi \in C_c^\infty(\Omega)$, we set
\begin{align}
\langle (\z, Du), \varphi \rangle := -\int_{\Omega} u \, \mathrm{div}(\varphi \z) \, dx &= - \int_\Omega u \, \varphi \, \mathrm{div}(\z) \, dx - \int_\Omega u \, \z \cdot \nabla \varphi \, dx \\
&= - \int_\Omega u \, \varphi \, \mathrm{div}(\z) \, dx - \sum_{i = 1}^k \int_\Omega u \, \z_i \cdot \nabla_{x_i} \varphi \, dx.
\end{align}
We also define
\begin{equation}
\langle (\z_1, D_{x_1} u), \varphi \rangle = \langle (\z, Du), \varphi \rangle - \sum_{i = 2}^k \int_\Omega \varphi \, \z_i \cdot \nabla_{x_i} u \, dx. 
\end{equation}
\end{definition}

Note that under condition \eqref{eq:conditionforuz} all the integrals are well-defined and finite. The newly defined objects $(\mathbf{z},Du)$ and $(\mathbf{z}_1, D_{x_1} u)$ a priori are distributions; in the next Proposition, we prove that they are actually Radon measures.

\begin{proposition}\label{prop:propertiesofanzellottipairing}
Assume that the pair $(u,\z)$ satisfies condition \eqref{eq:conditionforuz}. Then, the distribution $(\z,Du)$ is a Radon measure in $\Omega$. Moreover,
\begin{equation}
\bigg\vert \int_{B} (\z,Du) \bigg\vert \leq \| \z_1 \|_{\infty} \int_{B} \vert D_{x_1} u \vert + \sum_{i=2}^k \| \z_i \|_{p'_i} \| \nabla_{x_i} u \|_{p_i}
\end{equation}
for any Borel set $B \subseteq \Omega$. Moreover, $(\z_1, D_{x_1} u)$ is also a Radon measure and it satisfies
\begin{equation}
\bigg\vert \int_B (\z_1, D_{x_1} u) \bigg\vert \leq \| \z_1 \|_{\infty} \int_B |D_{x_1} u|
\end{equation}
for any Borel set $B \subseteq \Omega$, i.e. it is absolutely continuous with respect to $\vert D_{x_1} u \vert$.
\end{proposition}

\begin{proof}
For now, assume additionally that $u \in C^\infty(\Omega)$. We note that  $\varphi \z \in X_{q'}^{(p'_i)}(\Omega)$ for all $\varphi \in C_c^\infty(\Omega)$. Therefore, using the distributional definition of the divergence we get
\begin{align}
|\langle (\z, Du), \varphi \rangle| = \left|-\int_{\Omega} u \, \mathrm{div}(\varphi \z) \, dx \right| &=  \left|\int_{\Omega} \nabla u \cdot (\varphi \z) \, dx \right| = \left|\int_{\Omega} \varphi (\z \cdot \nabla u) \, dx \right| \\
&\leq \| \varphi \|_\infty \left| \int_{\Omega} \z \cdot \nabla u \, dx \right| \leq \| \varphi \|_\infty \bigg( \sum_{i = 1}^k \| \z_i \|_{p'_i}  \| \nabla_{x_i} u \|_{p_i} \bigg).
\end{align}
In the general case, assuming that $u \in BV^{(p_i)}(\Omega)$ satisfies the assumption \eqref{eq:conditionforuz}, take the sequence $u_j \in W^{1,1}(\Omega) \cap C^\infty(\Omega)$ given by the anisotropic Meyers-Serrin theorem (Proposition \ref{prop:meyersserrin}). Then, for any $\varphi \in C_c^\infty(\Omega)$ we get
$$
\lim_{j \rightarrow \infty} \langle (\z, Du_j), \varphi \rangle = \lim_{j \rightarrow \infty} -\int_{\Omega} u_j \, \mathrm{div}(\varphi \z) \, dx = - \int_{\Omega} u \, \mathrm{div}(\varphi \z)  \, dx  = \langle (\z, Du), \varphi \rangle.$$
As a consequence,
\begin{align}
|\langle (\z, Du), \varphi \rangle| = \lim_{j \rightarrow \infty} |\langle (\z, Du_j), \varphi \rangle| &\leq \lim_{j \rightarrow \infty} \| \varphi \|_\infty \bigg( \sum_{i = 1}^k \| \z_i \|_{p'_i}  \| \nabla_{x_i} u_j \|_{p_i} \bigg) \\
&= \| \varphi \|_\infty \bigg( \| \z_1 \|_\infty \int_\Omega |D_{x_1} u| + \sum_{i = 2}^k \| \z_i \|_{p'_i}  \| \nabla_{x_i} u \|_{p_i} \bigg).
\end{align}
Thus, $(\z, Du)$ is a continuous functional on the space of smooth functions (equipped with the supremum norm). Since smooth functions are dense in continuous functions in the supremum norm, $(\z, Du)$ defines a continuous functional on the space $C(\Omega)$. By the Riesz representation theorem, we get that $(\z, Du)$ is a Radon measure and
\begin{equation}
\bigg\vert \int_{B} (\z,Du) \bigg\vert \leq \| \z_1 \|_{\infty} \int_{B} \vert D_{x_1} u \vert + \sum_{i=2}^k \| \z_i \|_{p'_i} \| \nabla_{x_i} u \|_{p_i}.
\end{equation}
We now turn our attention to the second defined pairing, i.e. $(\z_1, D_{x_1} u)$. It is clearly a Radon measure, since $(\z,Du)$ is a Radon measure and the difference between these two objects is a function $(\sum_{i=2}^k \z_i \cdot \nabla_{x_i} u) \in L^1(\Omega)$. We make a similar computation and get that for any $u \in C^\infty(\Omega)$ and $\varphi \in C_c^\infty(\Omega)$ we have
\begin{align}\label{eq:boundforsecondpairing}
|\langle (\z_1, D_{x_1} u), \varphi \rangle| &= \left|-\int_{\Omega} u \, \mathrm{div}(\varphi \z) \, dx - \sum_{i=2}^k \varphi \, \int_\Omega \z_i \cdot \nabla_{x_i} u \, dx \right| \\ 
&=  \left|\int_{\Omega} \nabla u \cdot (\varphi \z) \, dx - \sum_{i=2}^k \int_\Omega \varphi \, \z_i \cdot \nabla_{x_i} u \, dx \right| \\
&= \left|\int_{\Omega} \varphi (\z_1 \cdot \nabla_{x_1} u) \, dx \right| \leq \| \varphi \|_\infty \| \z_1 \|_\infty  \| \nabla_{x_1} u \|_1.
\end{align}
In the general case, assuming that $u \in BV^{(p_i)}(\Omega)$ satisfies the assumption \eqref{eq:conditionforuz}, we again take the sequence $u_j \in W^{1,1}(\Omega) \cap C^\infty(\Omega)$ given by the anisotropic Meyers-Serrin theorem (Proposition \ref{prop:meyersserrin}). Then, for any $\varphi \in C_c^\infty(\Omega)$ we get
\begin{align}\label{eq:approximationforsecondpairing}
\lim_{j \rightarrow \infty} \langle (\z_1, D_{x_1} u_j), \varphi \rangle &= \lim_{j \rightarrow \infty} \bigg( -\int_{\Omega} u_j \, \mathrm{div}(\varphi \z) \, dx - \sum_{i=2}^k \int_\Omega \varphi \, \z_i \cdot \nabla_{x_i} u_j \, dx \bigg) \\
&= -\int_{\Omega} u \, \mathrm{div}(\varphi \z) \, dx - \sum_{i=2}^k \int_\Omega \varphi \, \z_i \cdot \nabla_{x_i} u \, dx = \langle (\z_1, D_{x_1} u), \varphi \rangle.
\end{align}
Consequently,
\begin{align}
|\langle (\z_1, D_{x_1} u), \varphi \rangle| &= \lim_{j \rightarrow \infty} |\langle (\z_1, D_{x_1} u_j), \varphi \rangle| \\
&\leq \lim_{j \rightarrow \infty} \| \varphi \|_\infty \| \z_1 \|_\infty  \| \nabla_{x_1} u_j \|_1 = \| \varphi \|_\infty \| \z_1 \|_\infty \int_\Omega |D_{x_1} u|,
\end{align}
and arguing similarly as above we get that
\begin{equation}
\bigg\vert \int_B (\z_1, D_{x_1} u) \bigg\vert \leq \| \z_1 \|_{\infty} \int_B |D_{x_1} u|, 
\end{equation}
so in particular we have $(\z_1, D_{x_1} u) \ll \vert D_{x_1} u \vert$, which concludes the proof.
\end{proof}

The proof of this Proposition justifies the notation $(\z_1, D_{x_1} u)$ in the following way. Given two vector fields $\z, \widetilde{\z}$ which satisfy condition \eqref{eq:conditionforuz} for some function $u$, if we have that $\z_1 = \widetilde{\z}_1$ a.e., then $(\z_1, D_{x_1} u) = (\widetilde{\z}_1, D_{x_1} u)$ as measures. Indeed, whenever $u \in C^\infty(\Omega)$, arguing as in \eqref{eq:boundforsecondpairing} for all $\varphi \in C_c^\infty(\Omega)$ we have
\begin{align}
\langle (\z_1, D_{x_1} u), \varphi \rangle = \int_{\Omega} \varphi (\z_1 \cdot \nabla_{x_1} u) \, dx = \int_{\Omega} \varphi (\widetilde{\z}_1 \cdot \nabla_{x_1} u) \, dx = \langle (\widetilde{\z}_1, D_{x_1} u), \varphi \rangle,
\end{align}
and by considering approximations $u_j$ of $u$ as in the anisotropic Meyers-Serrin theorem (Proposition \ref{prop:meyersserrin}) and arguing as in \eqref{eq:approximationforsecondpairing} we get that for all $\varphi \in C_c^\infty(\Omega)$
\begin{align}
\langle (\z_1, D_{x_1} u), \varphi \rangle = \lim_{j \rightarrow \infty} \langle (\z_1, D_{x_1} u_j), \varphi \rangle = \lim_{j \rightarrow \infty} \langle (\widetilde{\z}_1, D_{x_1} u_j), \varphi \rangle = \langle (\widetilde{\z}_1, D_{x_1} u), \varphi \rangle,
\end{align}
so the pairing $(\z_1, D_{x_1} u)$ only depends on the coordinate $\z_1$ of the vector field $\z$.

Now, let us observe that we have a co-area formula similar to the one given in Theorem \ref{thm:isotropiccoarea} for the measure $|D_{x_1} u|$. Recall that for any measurable $u: \Omega \rightarrow \R$ we denote $E_t = \{ u(x) > t \}$. Using the characterisation of the measure $|D_{x_1} u|$ given in \eqref{dfn:tvonx1}, working as in the proof given in \cite[Theorem 5.5.1]{EG} we obtain the following result.

\begin{theorem}\label{thm:anisotropiccoarea}
Let $u \in BV_{x_1}(\Omega)$. Then, for almost all $t \in \R$ we have that $\chi_{E_t} \in BV_{x_1}(\Omega)$ and
\begin{equation}\label{eq:anisotropiccoarea}
\int_\Omega |D_{x_1} u| = \int_{-\infty}^\infty \int_\Omega |D_{x_1} \chi_{E_t}| \, dt.
\end{equation}
Furthermore, whenever $u \in L^1(\Omega)$ is such that the expression on the right-hand side of \eqref{eq:anisotropiccoarea} is finite, we have that $u \in BV_{x_1}(\Omega)$.
\end{theorem}

Similarly to the standard BV case, for $u \in BV_{x_1}(\Omega)$ we denote by $\frac{D_{x_1} u}{|D_{x_1} u|}$ the Radon-Nikodym derivative of $D_{x_1 }u$ with respect to $|D_{x_1} u|$. As a consequence of the co-area formula given in Theorem \ref{thm:anisotropiccoarea}, we have that for almost all $t \in \R$ 
$$ \frac{D_{x_1} \chi_{E_t}}{|D_{x_1} \chi_{E_t}|} = \frac{D_{x_1} u}{|D_{x_1} u|} \quad |D_{x_1} \chi_{E_t}|-\mbox{a.e. in } \Omega.$$
A natural question is whether the strongly anisotropic Anzellotti pairing defined above satisfies an analogue of the co-area formula (see Theorem \ref{thm:isotropiccoarea} in the isotropic case). Unfortunately, the answer is negative: an analogue of the isotropic co-area formula for the strongly anisotropic Anzellotti pairing would be that
\begin{equation*}
\langle (\z_1, D_{x_1} u), \varphi \rangle = \int_{-\infty}^\infty \langle (\z_1, D_{x_1} \chi_{E_t}),\varphi \rangle \, dt
\end{equation*}
holds for all functions $\varphi \in C_c^\infty(\Omega)$. Unfortunately, the object on the right-hand side is not even well-defined, as in general for $u \in BV^{(p_i)}(\Omega)$ its characteristic function $\chi_{E_t}$ does not belong to the same space. Therefore, we now proceed to prove a weaker result (Proposition \ref{prop:lipschitzfunctionofdensity}) which will act as a replacement of the co-area formula in the proofs in Section \ref{sec:dirichlet}, in particular in Lemma \ref{lem:completeaccretivity}.

By Proposition \ref{prop:propertiesofanzellottipairing}, the measure $(\z_1 ,D_{x_1} u)$ is absolutely continuous with respect to the measure $|D_{x_1} u|$. By the Radon-Nikodym theorem, there exists a measurable function $\theta(\z_1,D_{x_1} u,x)$ which is the density of the measure $(\z_1,D_{x_1} u)$ with respect to $|D_{x_1} u|$, i.e. for all Borel sets $B \subset \Omega$ we have
\begin{equation}\label{Borel}
\int_B (\z_1, D_{x_1} u) = \int_B \theta(\z_1, D_{x_1} u,x) \, d|D_{x_1} u|.
\end{equation}
Moreover, by the estimate in Proposition \ref{prop:propertiesofanzellottipairing}, we have that 
$$|\theta(\z_1, D_{x_1} u,x)| \leq \| \z_1 \|_\infty \quad |D_{x_1} u|-\mbox{a.e. in } \Omega.$$

Taking a sequence of mollifications of a vector field $\z \in X_{q'}^{(p'_i)}(\Omega)$, we can easily prove the following result.

\begin{lemma}\label{lem:approximationofz}
For every $\z \in X_{q'}^{(p'_i)}(\Omega)$, there exists a sequence $\z^n \in C^\infty(\Omega; \mathbb{R}^N) \cap X_{q'}^{(p'_i)}(\Omega)$ with the following properties:

\begin{enumerate}
\item $\| \z_1^n \|_\infty \leq \| \z_1 \|_\infty$;

\item $\z_1^n \rightharpoonup \z_1$ weakly* in $L^\infty(\Omega; \mathbb{R}^{n_1})$ and $\z_1^n \rightarrow \z_1$ in $L^r_{\rm loc}(\Omega;\mathbb{R}^{n_1})$ for all $r \in [1,\infty)$;

\item $\z_i^n \rightarrow \z_i$ in $L^{p'_i}(\Omega;\mathbb{R}^{n_i})$ for all $i = 2, ..., k$;

\item $\z^n(x) \rightarrow \z(x)$ at every Lebesgue point $x$ of $\mathbf{z}$ and uniformly in sets of uniform continuity of $\z$;

\item $\mathrm{div}(\z^n) \rightarrow \mathrm{div}(\z)$ in $L^{q'}_{\rm loc}(\Omega)$.
\end{enumerate}
\end{lemma}

As a consequence, we get the following pointwise representation result for the density function $\theta(\z_1, D_{x_1} u, x)$.

\begin{proposition}\label{prop:pointwiseform}
Assume that the pair $(u,\z)$ satisfies condition \eqref{eq:conditionforuz} and suppose that $\z \in C(\Omega; \mathbb{R}^N)$. Then, we have
\begin{equation}\label{eq:pointwiseform}
\theta(\z_1, D_{x_1} u, x) = \z(x) \cdot \frac{D_{x_1} u}{|D_{x_1} u|}(x) \qquad |D_{x_1} u|-\mbox{\rm a.e. in } \Omega.
\end{equation}
\end{proposition}

\begin{proof}
By the definition of the Radon-Nikodym derivative $\frac{D_{x_1} u}{|D_{x_1} u|}$, condition \eqref{eq:pointwiseform} is equivalent to
\begin{equation}\label{eq:pointwiseform2}
\langle (\z_1, D_{x_1} u), \varphi \rangle = \int_\Omega \varphi \, \z_1 \, dD_{x_1} u \qquad \mbox{for all } \varphi \in C_c^\infty(\Omega).
\end{equation}
We first prove the claim for $\z \in C^1(\Omega; \mathbb{R}^N)$. Take a sequence $u_j \rightarrow u$ as given by the anisotropic Meyers-Serrin theorem (Proposition \ref{prop:meyersserrin}). By the distributional definition of the divergence, for all $\varphi \in C_c^\infty(\Omega)$ we have
\begin{align}
\langle (\z_1, D_{x_1} u_j), \varphi \rangle &= -\int_{\Omega} u_j \, \mathrm{div}(\varphi \z) \, dx - \sum_{i=2}^k \varphi \, \int_\Omega \z_i \cdot \nabla_{x_i} u_j \, dx  \\ 
&=  \int_{\Omega} \nabla u_j \cdot (\varphi \z) \, dx - \sum_{i=2}^k \int_\Omega \varphi \, \z_i \cdot \nabla_{x_i} u_j \, dx  = \int_{\Omega} \varphi (\z_1 \cdot \nabla_{x_1} u_j) \, dx.
\end{align}
By the continuity of $\z$, passing to the limit $j \rightarrow \infty$ we get that equation \eqref{eq:pointwiseform2} holds.

We now allow for general $\z \in C(\Omega; \mathbb{R}^N)$. Take a sequence of approximations $\z^n$ given by Lemma \ref{lem:approximationofz} and for any $\varphi \in C_c^\infty(\Omega)$ calculate
\begin{equation}
\langle (\z_1, D_{x_1} u), \varphi \rangle = \lim_{n \rightarrow \infty} \langle (\z_1^n, D_{x_1} u), \varphi \rangle = \lim_{n \rightarrow \infty} \int_\Omega \varphi \, \z_1^n \, dD_{x_1} u = \int_\Omega \varphi \, \z_1 \, dD_{x_1} u,
\end{equation}
where the last equality follows from continuity of $\z$ and uniform convergence of $\z^n$ to $\z$ on the support of $\varphi$.
\end{proof}

We now proceed to prove the result which will serve as a replacement of the co-area formula for the pairing $(\z_1, D_{x_1} u)$; we show that the Radon-Nikodym derivative $\theta$ is invariant under monotone Lipschitz transformations of the real line. This is a generalisation of \cite[Proposition 2.8]{Anzellotti} to the strongly anisotropic case; however, the proof is quite different from the original. This is because in \cite{Anzellotti} the result was given as a direct consequence of the co-area formula for the Anzellotti pairing, which is not valid in our setting.

\begin{proposition}\label{prop:lipschitzfunctionofdensity}
Assume that the pair $(u,\z)$ satisfies condition \eqref{eq:conditionforuz}. If $T: \mathbb{R} \rightarrow \mathbb{R}$ is a Lipschitz continuous increasing function, then
\begin{equation*}
\theta(\z_1, D_{x_1} (T \circ u),x) = \theta(\z_1, D_{x_1} u,x) \qquad |D_{x_1} u|-\mbox{\rm a.e. in } \Omega.
\end{equation*}
\end{proposition}

\begin{proof}
For $a,b \in \mathbb{R}$ with $a < b$, denote by $T_{a,b}(u)$ the truncation of $u$ at levels $a,b$, i.e.
\begin{equation*}
T_{a,b}(u) = \threepartdef{b}{u(x) \geq b;}{u(x)}{u(x) \in (a,b);}{a}{u(x) \leq a.}
\end{equation*}
Then, by Theorem \ref{thm:anisotropiccoarea} we have $D_{x_1} T_{a,b}(u) \in BV_{x_1}(\Omega)$ and $\int_{\Omega} |D_{x_1} T_{a,b}(u)| \leq \int_{\Omega} |D_{x_1} u|$. 

We first prove that for all $a,b \in \R$ we have
\begin{equation}\label{eq:equalitythetatruncation}
\theta(\z_1, D_{x_1} u,x) = \theta(\z_1, D_{x_1} T_{a,b}(u),x) \qquad |D_{x_1} T_{a,b}(u)|-\mbox{a.e. in } \Omega.
\end{equation}
Suppose otherwise; then, there exists a Borel set $B \subset \Omega$ such that $a \leq u(x) \leq b$ almost everywhere on $B$ and 
$$\theta(\z_1, D_{x_1} u,x) > \theta(\z_1, D_{x_1} T_{a,b}(u),x) \qquad |D_{x_1} T_{a,b}(u)|-\mbox{a.e. on } B;$$
the case when the opposite inequality holds is handled similarly. Hence,
\begin{align}\label{eq:truncationbycontradiction}
\int_B (\z_1, D_{x_1} u) &= \int_B \theta(\z_1, D_{x_1} u,x) |D_{x_1} u| = \int_B \theta(\z_1, D_{x_1} u,x) |D_{x_1} T_{a,b}(u)| \\
&> \int_B \theta(\z_1, D_{x_1} T_{a,b}(u),x) |D_{x_1} T_{a,b}(u)| = \int_B (\z_1, DT_{a,b}(u)).
\end{align}
Now, notice that
\begin{align}
\bigg| \int_B (\z_1, D_{x_1} u) - &\int_B (\z_1, D_{x_1} T_{a,b}(u)) \bigg| = \bigg| \int_B (\z_1, D_{x_1}(u - T_{a,b}(u))) \bigg| \\
&\leq \| \z_1 \|_\infty \int_B |D_{x_1}(u-T_{a,b}(u))| = \| \z_1 \|_{\infty} \int_{-\infty}^\infty \int_B |D_{x_1} \chi_{\{ u-T_{a,b}(u) \geq t \}}| \, dt \\
&= \| \z_1 \|_{\infty} \int_{-\infty}^a \int_B |D_{x_1} \chi_{E_t}| \, dt + \| \z_1 \|_{\infty} \int_{b}^\infty \int_B |D_{x_1} \chi_{E_t}| \, dt = 0,
\end{align}
since $a \leq u \leq b$ a.e. on $B$. This gives a contradiction with \eqref{eq:truncationbycontradiction}, so \eqref{eq:equalitythetatruncation} holds.

Now, take an approximating sequence $\z^n \in C^\infty(\Omega; \mathbb{R}^N) \cap X_{q'}^{(p'_i)}(\Omega)$ as given in Lemma \ref{lem:approximationofz}. Then, using Proposition \ref{prop:pointwiseform} and Theorem \ref{thm:anisotropiccoarea}, we get
\begin{align}
\langle (\z^n_1, D_{x_1} u), \varphi \rangle &= \int_\Omega \z_1^n(x) \cdot \frac{D_{x_1} u}{|D_{x_1} u|}(x) \, \varphi(x) \, d|D_{x_1} u| \\
&= \int_{-\infty}^{\infty} \bigg( \int_\Omega \z_1^n(x) \cdot \frac{D_{x_1} \chi_{E_t}}{|D_{x_1} \chi_{E_t}|}(x) \, \varphi(x) \, d|D_{x_1} \chi_{E_t}| \bigg) \, dt.
\end{align}
The sequence $\z_1^n \cdot \frac{D_{x_1} \chi_{E_t}}{|D_{x_1} \chi_{E_t}|}$ is bounded in the supremum norm on its domain of definition, i.e. $\z_1^n$ as a smooth function with $\| \z^n \|_\infty \leq 1$ is bounded by $1$ everywhere in $\Omega$ and by definition the function $\frac{D_{x_1} \chi_{E_t}}{|D_{x_1} \chi_{E_t}|}$ is bounded (also by $1$) $|D_{x_1} \chi_{E_t}|$-a.e. in $\Omega$. Denote by $\widetilde{\theta}(\z_1,u,t)$ any weak* limit of the sequence $\z_1^n \cdot \frac{D_{x_1} \chi_{E_t}}{|D_{x_1} \chi_{E_t}|}$ in $L^\infty(\Omega, |D_{x_1} \chi_{E_t}|)$.  

We now prove that $\widetilde{\theta}(\z_1,u,t)$ is uniquely defined. To this end, we compute
\begin{align}
\langle (\z^n_1, D_{x_1} T_{a,b} (u)), \varphi \rangle &= \int_\Omega \z_1^n(x) \cdot \frac{D_{x_1} T_{a,b}(u)}{|D_{x_1} T_{a,b}(u)|}(x) \, \varphi(x) \, d|D_{x_1} u| \\
&= \int_{a}^{b} \bigg( \int_\Omega \z_1^n(x) \cdot \frac{D_{x_1} \chi_{E_t}}{|D_{x_1} \chi_{E_t}|}(x) \, \varphi(x) \, d|D_{x_1} \chi_{E_t}| \bigg) \, dt
\end{align}
and pass to the limit with $n \rightarrow \infty$. By the dominated convergence theorem, we obtain
\begin{align}
\langle (\z_1, D_{x_1} T_{a,b} (u)), \varphi \rangle = \int_{a}^{b} \bigg( \int_\Omega \widetilde{\theta}(\z_1,u,t)(x) \, \varphi(x) \, d|D_{x_1} \chi_{E_t}| \bigg) \, dt.
\end{align}
Since by equation \eqref{eq:equalitythetatruncation} and Theorem \ref{thm:anisotropiccoarea} we can write the left-hand side in the following way
\begin{align}
\langle (\z_1, &D_{x_1} T_{a,b} (u)), \varphi \rangle = \int_\Omega \varphi \, d(\z_1,D_{x_1} T_{a,b}(u)) = \int_\Omega \theta(\z_1, D_{x_1} T_{a,b}(u),x) \, \varphi(x) \, d|D_{x_1} T_{a,b}(u)| \\
&= \int_\Omega \theta(\z_1, D_{x_1} u,x) \, \varphi(x) \, d|D_{x_1} T_{a,b}(u)| = \int_a^b \bigg( \int_\Omega \theta(\z_1, D_{x_1} u,x) \, \varphi(x) \, d|D_{x_1} \chi_{E_t}| \bigg) \, dt,
\end{align}
we get that
\begin{equation}
\int_a^b \bigg( \int_\Omega \theta(\z_1, D_{x_1} u,x) \, \varphi(x) \, d|D_{x_1} \chi_{E_t}| \bigg) \, dt = \int_{a}^{b} \bigg( \int_\Omega \widetilde{\theta}(\z_1,u,t)(x) \, \varphi(x) \, d|D_{x_1} \chi_{E_t}| \bigg) \, dt.
\end{equation}
Since $a$ and $b$ were arbitrary, we get that for almost all $t \in \mathbb{R}$ we have
\begin{equation}
\int_\Omega \theta(\z_1, D_{x_1} u,x) \, \varphi(x) \, d|D_{x_1} \chi_{E_t}| = \int_\Omega \widetilde{\theta}(\z_1,u,t)(x) \, \varphi(x) \, d|D_{x_1} \chi_{E_t}|,
\end{equation}
and since $f$ was arbitrary, by a density argument we get that for almost all $t \in \R$
\begin{equation}\label{eq:equalityofthetasforalmostallt}
\theta(\z_1, D_{x_1} u,x) = \widetilde{\theta}(\z_1,u,t)(x) \qquad |D_{x_1} \chi_{E_t}|-\mbox{a.e.}
\end{equation}
In particular, $\widetilde{\theta}(\z_1,u,t)$ is uniquely defined. 

To conclude the proof, notice that
\begin{equation*}
E_t = \{ x \in \Omega: \, u(x) > t \} = \{ x \in \Omega: \, (T \circ u)(x) > T(t) \},
\end{equation*}
so if we take any sequence $\z^n$ approximating $\z$ as in Lemma \ref{lem:approximationofz}, by the anisotropic co-area formula (Theorem \ref{thm:anisotropiccoarea}) we have
\begin{equation}
\widetilde{\theta}(\z_1,u,t)(x) = \mathrm{w}^*\mbox{-}\lim_{n \rightarrow \infty} \z_1^n \cdot \frac{D_{x_1} \chi_{E_t}}{|D_{x_1} \chi_{E_t}|} = \mathrm{w}^*\mbox{-}\lim_{n \rightarrow \infty} \z_1^n \cdot \frac{D_{x_1} \chi_{\{ T \circ u \geq t\}}}{|D_{x_1} \chi_{\{ T \circ u \geq t \}}|} = \widetilde{\theta}(\z_1,T(u),T(t))(x).
\end{equation}
Hence, by equation \eqref{eq:equalityofthetasforalmostallt}, for $\mathcal{L}^1$-almost all $t \in \mathbb{R}$
\begin{equation*}
\theta(\z_1, D_{x_1} u,x) = \widetilde{\theta}(\z_1,u,t)(x) = \widetilde{\theta}(\z_1,T(u),T(t))(x) = \theta(\z_1, D_{x_1}(T \circ u),x)
\end{equation*}
$|D_{x_1} \chi_{E_t}|$-a.e. in $\Omega$. By the anisotropic co-area formula (Theorem \ref{thm:anisotropiccoarea}), this equality also holds $|D_{x_1} u|$-a.e., which concludes the proof.
\end{proof}

Using an argument as in \cite[Proposition 2.7]{LS}, one can show that whenever the function $u$ satisfies the chain rule $D_{x_1} (f \circ u) = f'(u) D_{x_1} u$ for all Lipschitz functions $f: \R \rightarrow \R$, the result above extends to all nondecreasing Lipschitz functions $T: \R \rightarrow \R$ (with the desired property valid $|D_{x_1} (T \circ u)|$-a.e. in $\Omega$). 

We now move to the last topic concerning strongly anisotropic Anzellotti pairings, namely an anisotropic counterpart of the weak Gauss-Green formula given in Theorem \ref{thm:isotropicgreenformula}. We prove the result in a setting adapted to the operators we consider in Sections \ref{sec:dirichlet} and \ref{sec:elliptic}, i.e. for functions which have zero trace on $\Omega_1 \times \partial\Omega_2$. Such a result was first proved in a restricted setting adapted to equation \eqref{eq:introductionellipticproblem} in \cite[Theorem 2.7]{MRST}, but the proof contains a minor flaw. Namely, in the current notation, the authors use integration by parts separately in the $x_1$ and $x_2$ variables, but only integrability of $\mathrm{div}(\z)$ is assumed; therefore, we do not know if $\mathrm{div}_{x_1}(\z)$ and $\mathrm{div}_{x_2}(\z)$ are integrable, which is needed to apply the integration by parts. To rectify this issue, we prove the anticipated anisotropic Gauss-Green formula in Theorem \ref{thm:anisotropicgaussgreen} from scratch in the following way: we first prove existence of the weak normal trace of a vector field with integrable divergence in Theorem \ref{thm:definitionofweaktrace}, from which follows the Gauss-Green formula for Sobolev functions, and then use an approximation as in the anisotropic Meyers-Serrin theorem (Proposition \ref{prop:meyersserrin})  to conclude the proof in the general case. The heart of the proof lies in Proposition \ref{prop:bilinearform} and Theorem \ref{thm:definitionofweaktrace}, and a careful examination of the proofs shows that we only use the integrability assumption on the full divergence $\mathrm{div}(\z)$.

As a preparation, we now prove a slightly improved version of the Gagliardo extension theorem.

\begin{proposition}\label{prop:gagliardo}
Let $f \in L^1(\partial\Omega_1 \times \Omega_2)$. Then, for every $\varepsilon > 0$ there exists a function $u \in W_{0,w}^{1,(p_i)}(\Omega)$ such that $u = f$ on $\partial\Omega_1 \times \Omega_2$,
$$\mathrm{supp}(u) \subset \{ x \in \Omega: \, \mathrm{dist}(x, \partial\Omega_1 \times \Omega_2) \leq \varepsilon \},$$
and the following inequalities hold: on the first coordinate, we have
\begin{equation}\label{eq:extensionestimatepart1}
\int_\Omega |D_{x_1} u| \leq (1 + \varepsilon) \| f \|_{L^1(\partial\Omega_1 \times \Omega_2)};
\end{equation}
and for all $i = 2, ..., k$ we have
\begin{equation}\label{eq:extensionestimatepart2}
\int_\Omega |\nabla_{x_i} u|^{p_i} \leq \varepsilon \| f \|_{L^1(\partial\Omega_1 \times \Omega_2)}.
\end{equation}

\end{proposition}

\begin{proof}
The result follows by modifying slightly the classical construction of the extension $u \in W^{1,1}(\Omega)$ of $f \in L^1(\partial\Omega)$ given by the Gagliardo extension theorem (see for instance \cite[Theorem 2.16]{Giu}). Since $\partial\Omega$ is Lipschitz and we only need to extend the boundary datum in a neighbourhood of $\partial\Omega_1 \times \Omega_2$, using an argument based on a partition of unity and a straightening of the boundary we can reduce the proof to the case when $\partial\Omega_1 \times \Omega_2 = \mathbb{R}^{N-1}$, $f$ has compact support in $\mathbb{R}^{N-1}$, and $u$ is a function defined in $\R^N_+ := \{ (y_1, ..., y_N): y_1 > 0 \}$. Note that we choose a slightly different notation for points: $(y_1, ..., y_N)$ has $N$ coordinates instead of $k$, $x_1$ corresponds to $(y_1, ..., y_{n_1})$, and for $i = 2,...,k$ the coordinates $x_i$ correspond to $(y_{n_1 + ... + n_{i-1} + 1}, ..., y_{n_1 + ... + n_i})$.

We first pick a sequence of smooth functions $f_j \in C_c(\R^{N-1})$ which converges to $f$ in $L^1(\R^{N-1})$ as $j \rightarrow \infty$. We can assume that $f_0 \equiv 0$ and
\begin{equation}\label{eq:finitesumfortheapproximation}
\sum_{j = 0}^\infty \| f_j - f_{j+1} \|_{L^1(\R^{N-1})} < \infty.
\end{equation}
Since $f_j$ have compact support, for every $j \in \mathbb{N} \cup \{ 0 \}$ we have that
\begin{equation}\label{eq:definitionofgj}
g_j := \sum_{l = 2}^{n_1} \int_{\R^N_+} \bigg( \bigg| \frac{\partial}{\partial y_l} f_j \bigg| + \bigg| \frac{\partial}{\partial y_l} f_{j+1} \bigg| \bigg) \, dx < \infty 
\end{equation}
and
\begin{equation}\label{eq:definitionofhj}
h_j := \max_{i = 2,...,k} \int_{\R^N_+} (|\nabla_{x_i} f_j|^{p_i} + |\nabla_{x_i} f_{j+1}|^{p_i}) \, dx < \infty. 
\end{equation}
Take a decreasing sequence $t_j$ converging to zero; we will fix the exact values of $t_j$ at the end of the proof. Denote the $y_1$ variable by $t$ and set
\begin{equation}
u(t,y') = \twopartdef{0}{\mbox{if } t > t_0;}{\frac{t - t_{j+1}}{t_j - t_{j+1}} f_j(y') + \frac{t_j - t}{t_j - t_{j+1}} f_{j+1}(y')}{\mbox{if } t \in [t_{j+1},t_{j}]}
\end{equation}
for $t > 0$ and $y' \in \mathbb{R}^{N-1}$. The argument in \cite{Giu} shows that the trace of $u$ is correct; we only need to prove the desired bounds \eqref{eq:extensionestimatepart1} and \eqref{eq:extensionestimatepart2}.

To this end, observe that for $t \in [t_{j+1}, t_j]$ we have the following pointwise bounds:
\begin{equation}
\bigg|\frac{\partial}{\partial t} u(t,y') \bigg| \leq |f_j(y') - f_{j+1}(y')| (t_j - t_{j+1})^{-1};
\end{equation}
for all $l = 2, ..., n_1$ we have
\begin{equation}
\bigg|\frac{\partial}{\partial y_l} u(t,y') \bigg| \leq \bigg|\frac{\partial}{\partial y_l} f_j(y') \bigg| + \bigg|\frac{\partial}{\partial y_l} f_{j+1}(y') \bigg|;
\end{equation}
and for all $i = 2, ..., k$ we have
\begin{equation}
| \nabla_{x_i} u(t,y') | \leq | \nabla_{x_i} f_j(y') | + | \nabla_{x_i} f_{j+1}(y') |.
\end{equation}
We will show that the desired estimates follow. To prove \eqref{eq:extensionestimatepart1}, observe that
\begin{align}
|\nabla_{x_1} u| \leq \bigg|\frac{\partial}{\partial t} u \bigg| + \sum_{l = 2}^{n_1} \bigg|\frac{\partial}{\partial y_l} u \bigg| \leq |f_j - f_{j+1}| (t_j - t_{j+1})^{-1} + \sum_{l = 2}^{n_1} \bigg( \bigg|\frac{\partial}{\partial y_l} f_j \bigg| + \bigg|\frac{\partial}{\partial y_l} f_{j+1} \bigg| \bigg),
\end{align}
and integrating this inequality over $\R^N_+$ we get
\begin{align}\label{eq:extensionestimateproof1}
\int_{\R^N_+} |\nabla_{x_1} u| \, dx \leq \|f_j - f_{j+1} \|_{L^1(\R^{N-1})} + (t_j - t_{j+1}) \, g_j,
\end{align}
where $g_j$ is given by \eqref{eq:definitionofgj}. Similarly, for all $i = 2, ..., k$ we have
\begin{equation}
|\nabla_{x_i} u|^{p_i} \leq (| \nabla_{x_i} f_j | + | \nabla_{x_i} f_{j+1} |)^{p_i} \leq 2^{p_i - 1} (| \nabla_{x_i} f_j|^{p_i} + | \nabla_{x_i} f_{j+1} |^{p_i}),
\end{equation}
and integrating this inequality over $\R^N_+$ we get
\begin{align}\label{eq:extensionestimateproof2}
\int_{\R^N_+} |\nabla_{x_i} u|^{p_i} \, dx \leq 2^{p_i - 1} \, (t_j - t_{j+1}) \, h_j,
\end{align}
where $h_j$ is given by \eqref{eq:definitionofhj}. Choosing the sequence $t_j$ in such a way that $t_0 < \varepsilon$ and
\begin{equation}
t_j - t_{j+1} \leq \frac{\| f \|_{L^1(\R^{N-1})}}{1 + g_j + 2^{p_k - 1} h_j} 2^{-j-2} \varepsilon,
\end{equation}
we obtain that the bound on the support holds and the estimates \eqref{eq:extensionestimatepart1} and \eqref{eq:extensionestimatepart2} follow from \eqref{eq:extensionestimateproof1} and \eqref{eq:extensionestimateproof2} respectively.
\end{proof}

We now prove that there exists a function $[\z_1, \nu_{x_1}]$ which has an interpretation of a weak normal trace of the vector field $\z \in X_1^{(p'_i)}(\Omega)$ on $\partial \Omega_1 \times \Omega_2$. To simplify the notation, we denote
$$ BV_{0,w}^{(p_i), \infty}(\Omega) = BV_{0,w}^{(p_i)}(\Omega) \cap L^\infty(\Omega).$$
The proof follows in two steps: in Proposition \ref{prop:bilinearform} we introduce an auxiliary pairing
$$ \langle \mathbf{z}, u \rangle_{\partial\Omega_1 \times \Omega_2}: X_1^{(p'_i)}(\Omega) \times BV_{0,w}^{(p_i), \infty}(\Omega) \rightarrow \mathbb{R}$$
and then in Theorem \ref{thm:definitionofweaktrace} we provide its integral representation, from which we deduce existence of a function in $L^\infty(\partial\Omega_1 \times \Omega_2)$ which has an interpretation of a weak normal trace of the vector field $\z$.

\begin{proposition}\label{prop:bilinearform}
There exists a bilinear map $\langle \mathbf{z}, u \rangle_{\partial\Omega_1 \times \Omega_2}: X_1^{(p'_i)}(\Omega) \times BV_{0,w}^{(p_i), \infty}(\Omega) \rightarrow \mathbb{R}$ such that
$$ \langle \mathbf{z}, u \rangle_{\partial\Omega_1 \times \Omega_2} = \int_{\partial\Omega_1 \times \Omega_2} u \, \mathbf{z} \cdot \nu^\Omega \, d\mathcal{H}^{N-1} \qquad \mbox{ if } \mathbf{z} \in C^1(\overline{\Omega}; \mathbb{R}^N),$$
where $\nu^\Omega$ denotes the outer unit normal to $\Omega$, and
$$ |\langle \mathbf{z}, u \rangle_{\partial\Omega_1 \times \Omega_2}| \leq \| \z_1 \|_{\infty} \cdot \| u \|_{L^1(\partial\Omega_1 \times \Omega_2)}.$$
\end{proposition}

\begin{proof}
For all $\mathbf{z} \in X_1^{(p'_i)}(\Omega)$ and $u \in BV_{0,w}^{(p_i), \infty}(\Omega) \cap W^{1,1}(\Omega)$ we set
\begin{equation}\label{eq:definitionofbilinearform}
\langle \mathbf{z}, u \rangle_{\partial\Omega_1 \times \Omega_2} = \int_\Omega u \, \mathrm{div}(\mathbf{z}) \, dx + \int_\Omega \mathbf{z} \cdot \nabla u \, dx.    
\end{equation}
Since we assume in the definition of $BV_{0,w}^{(p_i)}(\Omega)$ that the trace of $u$ on $\Omega_1 \times \partial\Omega_2$ equals zero, by Theorem \ref{thm:standardtracetheorem} the first property holds. Moreover, this map is bilinear. 

In the general case, due to the fact that $D_{x_1} u$ may be only a measure, the formula above is not well-defined; we will extend it by approximating general $u \in BV_{0,w}^{(p_i), \infty}(\Omega)$ using smooth functions. To this end, we notice that if $u, v \in BV_{0,w}^{(p_i), \infty}(\Omega) \cap W^{1,1}(\Omega)$ have the same trace, then
\begin{equation}\label{eq:equalityforthesametrace}
\langle \mathbf{z}, u \rangle_{\partial\Omega_1 \times \Omega_2} = \langle \mathbf{z}, v \rangle_{\partial\Omega_1 \times \Omega_2}.
\end{equation}
To prove this, consider an approximation $g_j$ of the function $u - v$ by smooth functions given by the anisotropic Meyers-Serrin theorem (Proposition \ref{prop:meyersserrin}). Since $u-v$ has trace zero, with a minor modification of the proof we can require that $u-v$ has compact support in $\Omega$. Then,
\begin{align}
\langle \mathbf{z}, u - v \rangle_{\partial\Omega_1 \times \Omega_2} &= \int_\Omega (u - v) \, \mathrm{div}(\mathbf{z}) \, dx + \int_\Omega \mathbf{z} \cdot \nabla(u - v) \, dx \\
&= \lim_{j \rightarrow \infty} \bigg( \int_\Omega g_j \, \mathrm{div}(\mathbf{z}) \, dx + \int_\Omega \mathbf{z} \cdot \nabla g_j \, dx\bigg) = 0,
\end{align}
where the last inequality follows from the distributional definition of the divergence; this concludes the proof of property \eqref{eq:equalityforthesametrace}. Since by the anisotropic Meyers-Serrin theorem (Proposition \ref{prop:meyersserrin}) for every $u \in BV_{0,w}^{(p_i)}(\Omega)$ there exists a smooth function with the same trace, for arbitrary $u \in BV_{0,w}^{(p_i), \infty}(\Omega)$ we define $\langle \mathbf{z}, u \rangle_{\partial\Omega_1 \times \Omega_2}$ by
$$ \langle \mathbf{z}, u \rangle_{\partial\Omega_1 \times \Omega_2} = \langle \mathbf{z}, v \rangle_{\partial\Omega_1 \times \Omega_2},$$
where $v$ is any function in $BV_{0,w}^{(p_i), \infty}(\Omega) \cap W^{1,1}(\Omega)$ with the same trace as $u$. In view of equation \eqref{eq:equalityforthesametrace}, this uniquely defines $\langle \mathbf{z}, u \rangle_{\partial\Omega_1 \times \Omega_2}$ for any $u \in BV_{0,w}^{(p_i), \infty}(\Omega)$.

Now, we have to prove the second property. Let us take a sequence $u_j \in BV_{0,w}^{(p_i), \infty}(\Omega) \cap C^\infty(\Omega)$ which converges to $u$ as in the anisotropic Meyers-Serrin theorem (Proposition \ref{prop:meyersserrin}). Then, we get that
\begin{align}
|\langle \mathbf{z}, u \rangle_{\partial\Omega_1 \times \Omega_2}| = |\langle \mathbf{z}, u_j \rangle_{\partial\Omega_1 \times \Omega_2}| &= \bigg| \int_\Omega u_j \, \mathrm{div}(\mathbf{z}) \, dx + \int_\Omega \mathbf{z} \cdot \nabla u_j \, dx \bigg| \\
&\leq \bigg| \int_\Omega u_j \, \mathrm{div}(\mathbf{z}) \, dx \bigg| +  \sum_{i = 1}^k \| \z_i \|_{p'_i}  \| \nabla_{x_i} u_j \|_{p_i}.
\end{align}
We pass to the limit with $j \rightarrow \infty$ and obtain
\begin{equation}\label{eq:estimateforpairingzu}
|\langle \mathbf{z}, u \rangle_{\partial\Omega_1 \times \Omega_2}| \leq \bigg| \int_\Omega u \, \mathrm{div}(\mathbf{z}) \, dx \bigg| + \| \z_1 \|_\infty \int_\Omega |D_{x_1} u| + \sum_{i = 2}^k \| \z_i \|_{p'_i}  \| \nabla_{x_i} u \|_{p_i}.    
\end{equation}
Fix $\varepsilon > 0$. Observe that by property \eqref{eq:equalityforthesametrace}, we may take $u$ to be the function given by the variant of the Gagliardo extension theorem proved in Proposition \ref{prop:gagliardo}; therefore,
$$ \int_\Omega |\nabla_{x_1} u| \, dx \leq (1 + \varepsilon) \| u \|_{L^1(\partial\Omega_1 \times \Omega_2)},$$
for $i = 2, ..., k$ we have
$$ \int_\Omega |\nabla_{x_i} u|^{p_i} \, dx \leq \varepsilon \| u \|_{L^1(\partial\Omega_1 \times \Omega_2)},$$
and $u$ is supported in $\Omega \backslash \Omega_\varepsilon$, where
$$ \Omega_\varepsilon = \{ x \in \Omega: \, \mathrm{dist}(x, \partial\Omega_1 \times \Omega_2) > \varepsilon \}.$$
We insert it in the estimate \eqref{eq:estimateforpairingzu} and obtain
\begin{align}
|\langle \mathbf{z}, u \rangle_{\partial\Omega_1 \times \Omega_2}| \leq \| u \|_\infty \bigg| \int_{\Omega \backslash \Omega_\varepsilon} \mathrm{div}(\mathbf{z}) \, dx \bigg| + (1 + \varepsilon) &\| \z_1 \|_{\infty} \| u \|_{L^1(\partial\Omega_1 \times \Omega_2)} \\
&+ \sum_{i=2}^k \varepsilon^{1/p_i} \, \| \z_i \|_{p'_i} \, \| u \|_{L^1(\partial\Omega_1 \times \Omega_2)}^{1/p_i}.    
\end{align}
Since $\varepsilon$ was arbitrary, we pass to the limit $\varepsilon \rightarrow 0$ and obtain
$$ |\langle \mathbf{z}, u \rangle_{\partial\Omega_1 \times \Omega_2}| \leq \| \mathbf{z_1} \|_{\infty} \| u \|_{L^1(\partial\Omega_1 \times \Omega_2)}, $$
which concludes the proof.
\end{proof}

Now, we provide an integral representation of the bilinear map $\langle \mathbf{z}, u \rangle_{\partial\Omega_1 \times \Omega_2}$, from which follows that for every vector field $\z \in X_1^{(p'_i)}$ there exists a function in $L^\infty(\partial\Omega_1 \times \Omega_2)$ which has an interpretation of its normal trace.

\begin{theorem}\label{thm:definitionofweaktrace}
There exists a linear operator $\gamma: X_1^{(p'_i)}(\Omega) \rightarrow L^\infty(\partial\Omega_1 \times \Omega_2)$ such that
$$ \| \gamma(\mathbf{z}) \|_{L^\infty(\partial\Omega_1 \times \Omega_2)} \leq \| \z_1 \|_{\infty},$$
we have the following integral representation: for all $u \in BV_{0,w}^{(p_i), \infty}(\Omega)$ 
\begin{equation}\label{eq:integralrepresentation}
\langle \mathbf{z}, u \rangle_{\partial\Omega_1 \times \Omega_2} = \int_{\partial\Omega_1 \times \Omega_2} u \, \gamma(\mathbf{z}) \, d\mathcal{H}^{N-1},    
\end{equation}
and
$$ \gamma(\mathbf{z}) (x) = \mathbf{z} \cdot \nu^\Omega \qquad \mbox{ if } \mathbf{z} \in C^1(\overline{\Omega}; \mathbb{R}^N).$$
The function $\gamma(\mathbf{z})$ is a weakly defined normal trace of $\mathbf{z}$ on $\partial\Omega_1 \times \Omega_2$; for this reason, we will denote it by $[\z_1, \nu_{x_1}]$.
\end{theorem}

\begin{proof}
Given $\mathbf{z} \in X_1^{(p'_i)}(\Omega)$, consider the linear functional $G: L^\infty(\partial\Omega_1 \times \Omega_2) \rightarrow \mathbb{R}$ defined by the formula
$$ G(f) = \langle \mathbf{z}, u \rangle_{\partial\Omega_1 \times \Omega_2},$$
where $u \in BV_{0,w}^{(p_i), \infty}(\Omega)$ is such that $u = f$ on $\partial\Omega_1 \times \Omega_2$. By Proposition \ref{prop:bilinearform}, we have
$$ |G(f)| = |\langle \mathbf{z}, u \rangle_{\partial\Omega_1 \times \Omega_2}| \leq \| \mathbf{z}_1 \|_{\infty} \| f \|_{L^1(\partial\Omega_1 \times \Omega_2)}.$$
Since $G$ is a continuous functional on (a dense subset of) $L^1(\partial\Omega_1 \times \Omega_2)$, there exists a function $\gamma(\mathbf{z}) \in L^\infty(\partial\Omega_1 \times \Omega_2)$ with norm at most equal to $\| \z_1 \|_{\infty}$ such that
$$ G(f) = \int_{\partial\Omega_1 \times \Omega_2} f \, \gamma(\mathbf{z}) \, d\mathcal{H}^{N-1},$$
which concludes the proof.
\end{proof}

\begin{corollary}\label{cor:greenformulaforsobolevfunctions}
For all $\z \in X_{q'}^{(p'_i)}(\Omega)$ and $u \in W_{0,w}^{1,(p_i)}(\Omega) \cap L^q(\Omega)$ we have
\begin{equation}
\int_{\Omega} u \, \mathrm{div}(\z) \, dx + \int_\Omega \z \cdot \nabla u \, dx = \int_{\partial\Omega_1 \times \Omega_2} u \, [\z_1, \nu_{x_1}] \,d\mathcal{H}^{N-1}.
\end{equation}
\end{corollary}

\begin{proof}
Since $\z \in X_{q'}^{(p'_i)}(\Omega)$, we also have that $\z \in X_{1}^{(p'_i)}(\Omega)$ and in particular the weak normal trace $[\z_1, \nu_{x_1}]$ is well-defined. Take a sequence $u_j$ which approximates $u$ as in the anisotropic Meyers-Serrin theorem (Proposition \ref{prop:meyersserrin}). By considering truncations, we can assume that $u_n$ is bounded; then, the functions may be no longer smooth or satisfy the trace constraint, but we have $u_j \in W_{0,w}^{1,(p_i)}(\Omega) \cap L^\infty(\Omega)$ and the sequence still satisfies
$$ u_j \rightarrow u \quad \mbox{in } L^{q}(\Omega)$$ 
and
$$\nabla_{x_i} u_j \rightarrow \nabla_{x_i} u \quad \mbox{in } L^{p_i}(\Omega; \mathbb{R}^{n_i})$$ 
for all $i = 1, ..., k$ (for $i = 1$ this follows from the assumption that $u \in W^{1,(p_i)}(\Omega)$). Then, by the definition of the bilinear form $\langle \z, u \rangle_{\partial\Omega_1 \times \Omega_2}$ given in equation \eqref{eq:definitionofbilinearform} and the integral representation \eqref{eq:integralrepresentation}, we have that
\begin{equation}
\int_{\Omega} u_j \, \mathrm{div}(\z) \, dx + \int_\Omega \z \cdot \nabla u_j \, dx = \int_{\partial\Omega_1 \times \Omega_2} u_j \, [\z_1, \nu_{x_1}] \,d\mathcal{H}^{N-1}.
\end{equation}
Passing to the limit $j \rightarrow \infty$, we obtain the desired result: convergence on the left-hand side follows from our assumption on the sequence $u_j$ and on the right-hand side from the fact that the trace operator is continuous with respect to convergence in $W^{1,1}(\Omega)$.
\end{proof}

Before we prove the anisotropic Gauss-Green formula, we need one additional technical result concerning the pairing $(\z_1, D_{x_1} u)$.

\begin{lemma}\label{lem:approximationofpairing}
Assume that the pair $(u,\z)$ satisfies condition \eqref{eq:conditionforuz}. Let $u_j \in C^\infty(\Omega) \cap BV^{(p_i)}(\Omega)$ converge to $u \in BV^{(p_i)}(\Omega)$ as in the anisotropic Meyers-Serrin theorem (Proposition \ref{prop:meyersserrin}). Then, we have
$$ \int_\Omega (\z_1, D_{x_1} u_j) \rightarrow \int_\Omega (\z_1, D_{x_1} u).$$
\end{lemma}

\begin{proof}
Fix $\varepsilon > 0$ and choose an open set $A \Subset \Omega$ such that 
$$\int_{\Omega \backslash A} |D_{x_1} u| < \varepsilon.$$
Let $g \in C_c^\infty(\Omega)$ be such that $0 \leq g \leq 1$ in $\Omega$ and $g \equiv 1$ in $A$. We write $1 = g + (1 - g)$ and estimate
\begin{align}\label{eq:approximationonthewholeset}
\bigg| \int_\Omega (\z_1, D_{x_1} u_j) - \int_\Omega (\z_1, D_{x_1} u) \bigg| \leq \bigg| \langle (\z_1, &D_{x_1} u_j), g \rangle -  \langle (\z_1, D_{x_1} u), g \rangle  \bigg| \\
&+ \int_\Omega |(\z_1, D_{x_1} u_j)| (1-g) + \int_\Omega |(\z_1, D_{x_1} u)| (1-g).
\end{align}
We already proved in \eqref{eq:approximationforsecondpairing} that for any $g \in C_c^\infty(\Omega)$ we have $\langle (\z_1, D_{x_1} u_j), g \rangle \rightarrow \langle (\z_1, D_{x_1} u), g \rangle$. Moreover, we have 
$$\int_\Omega (1-g) |(\z_1, D_{x_1} u)| \leq \int_{\Omega \backslash A} |(\z_1, D_{x_1} u)| \leq \| \mathbf{z}_1 \|_\infty \int_{\Omega \backslash A} |D_{x_1} u| < \varepsilon \| \mathbf{z}_1 \|_\infty$$
and similarly
$$\limsup_{j \rightarrow \infty} \int_\Omega (1-g) |(\z_1, D_{x_1} u_j)| \leq \limsup_{j \rightarrow \infty} \| \mathbf{z}_1 \|_\infty \int_{\Omega \backslash A} |D_{x_1} u_j| \leq \varepsilon \| \mathbf{z}_1 \|_\infty ,$$
so the right-hand side of \eqref{eq:approximationonthewholeset} goes to zero as $j \rightarrow \infty$.
\end{proof}

We conclude by proving the anisotropic counterpart of the Gauss-Green formula, which relates the measure $(\z, Du)$ (or equivalently: the measure $(\z_1, D_{x_1} u)$ and the pointwise products $\z_i \cdot \nabla_{x_i} u$ for $i = 2, ..., k$) with the weak normal trace $[\z_1, \nu_{x_1}]$.

\begin{theorem}[Anisotropic Gauss-Green formula]\label{thm:anisotropicgaussgreen}
For all functions $u \in BV_{0,w}^{(p_i)}(\Omega)$ and vector fields $\z \in X_{q'}^{(p'_i)}(\Omega)$ we have
\begin{equation}
\int_{\Omega} u \, \mathrm{div}(\z) \, dx + \int_\Omega (\z_1, D_{x_1} u) + \sum_{i=2}^k \int_\Omega \z_i \cdot \nabla_{x_i} u \, dx = \int_{\partial\Omega_1 \times \Omega_2} u \, [\z_1, \nu_{x_1}] \,d\mathcal{H}^{N-1}.
\end{equation}
\end{theorem}

\begin{proof}
Take an approximation $u_j \in C^\infty(\Omega) \cap W_{0,w}^{1,(p_i)}(\Omega)$ of $u$ given by the anisotropic Meyers-Serrin theorem (Proposition \ref{prop:meyersserrin}). Then, by Corollary \ref{cor:greenformulaforsobolevfunctions} we have
\begin{equation}
\int_{\Omega} u_j \, \mathrm{div}(\z) \, dx + \int_\Omega \z \cdot \nabla u_j \, dx = \int_{\partial\Omega_1 \times \Omega_2} u_j \, [\z_1, \nu_{x_1}] \,d\mathcal{H}^{N-1}.
\end{equation}
We now pass to the limit separately in each term. Since $u_j \rightarrow u$ in $L^q(\Omega)$ and $\mathrm{div}(\z) \in L^{q'}(\Omega)$, we have
\begin{equation}
\lim_{j \rightarrow \infty} \int_{\Omega} u_j \, \mathrm{div}(\z) \, dx = \int_{\Omega} u \, \mathrm{div}(\z) \, dx.
\end{equation}
By Lemma \ref{lem:approximationofpairing} and the strong convergence $\nabla_{x_i} u_j \rightarrow \nabla_{x_i} u$ in $L^{p_i}(\Omega; \mathbb{R}^{n_i})$ for all $i = 2, ..., k$, we have
$$ \lim_{j \rightarrow \infty} \int_\Omega \z \cdot \nabla u_j \, dx = \int_\Omega (\z_1, D_{x_1} u) + \sum_{i=2}^k \int_\Omega \z_i \cdot \nabla_{x_i} u \, dx.$$
Finally, since $u_j|_{\partial\Omega_1 \times \Omega_2} = u|_{\partial\Omega_1 \times \Omega_2}$, we have
$$ \lim_{j \rightarrow \infty} \int_{\partial\Omega_1 \times \Omega_2} [\z_1, \nu_{x_1}] \, u_j \, d\mathcal{H}^{N-1} = \int_{\partial\Omega_1 \times \Omega_2} [\z_1, \nu_{x_1}] \, u \, d\mathcal{H}^{N-1},$$
which concludes the proof.
\end{proof}

\begin{remark}
One can prove most of the results in this Section under a slightly weaker assumption on the vector fields and a slightly stronger assumption on the functions, namely in place of condition \eqref{eq:conditionforuz} assume that
\begin{equation}
u \in BV^{(p_i)}(\Omega) \cap L^{p_k}(\Omega) \cap C(\Omega) \quad \mbox{and} \quad \z \in X^{(p'_i)}_\mu(\Omega),
\end{equation}
where
\begin{equation}
X^{(p'_i)}_\mu(\Omega) := \bigg\{ \mathbf{z} \in L^{p'_1}(\Omega; \R^{n_1}) \times ... \times L^{p'_k}(\Omega; \R^{n_k}): \, \mathrm{div}(\mathbf{z}) \in \mathcal{M}(\Omega) \bigg\}.
\end{equation}
This corresponds to assumption (c) from the classical paper \cite{Anzellotti} due to Anzellotti. Under these assumptions, we can define $(\z, Du)$ and $(\z_1, D_{x_1} u)$ as in Definition \ref{dfn:anzellottipairing}, and the following results remain true with only minor modifications of the proofs (adding the zero-trace assumption or continuity of the function $u$ when necessary): Proposition \ref{prop:propertiesofanzellottipairing}; Proposition \ref{prop:bilinearform}; Theorem \ref{thm:definitionofweaktrace}; Corollary \ref{cor:greenformulaforsobolevfunctions}; Lemma \ref{lem:approximationofpairing}; and the anisotropic Gauss-Green formula (Theorem \ref{thm:anisotropicgaussgreen}).
\end{remark}

\section{The homogeneous Dirichlet problem}\label{sec:dirichlet}

Recall that throughout the whole paper we assume that $\Omega = \Omega_1 \times \Omega_2$ is a bounded Lipschitz domain in $\mathbb{R}^N$, where $\Omega_1 \subset \R^{n_1}$ and $\Omega_2 \subset \R^{n_2+...+n_k}$. Similarly to the previous Section, we assume that $p_1 = 1$ and $1 < p_2 \leq ... \leq p_k$: this corresponds to the fact that the anisotropic $p$-Laplacian operator we will study in this section has linear growth in some coordinates. Consider the homogeneous Dirichlet problem
\begin{equation}\label{eq:mainproblem}
\left\{ \begin{array}{lll} u_t (t,x) = \Delta_{(p_i)} u   \quad &\hbox{in} \ \ (0, T) \times \Omega; \\[5pt] u(t) = 0 \quad &\hbox{on} \ \ (0, T) \times  \partial\Omega; \\[5pt] u(0,x) = u_0(x) \quad & \hbox{in} \ \  \Omega, \end{array} \right.
\end{equation}
where $u_0 \in L^2(\Omega)$. Here, $\Delta_{(p_i)}$ denotes the anisotropic $p$-Laplacian operator, i.e.
\begin{equation}
\Delta_{(p_i)}(u) = \mathrm{div}_{x_1} \bigg( \frac{D_{x_1} u}{|D_{x_1} u|} \bigg) + \sum_{i = 2}^k \mathrm{div}_{x_i} (|\nabla_{x_i} u|^{p_i - 2} \nabla_{x_i} u).
\end{equation}
Consider the energy functional $\mathcal{F} : L^2(\Omega) \rightarrow [0, + \infty]$ associated with problem \eqref{eq:mainproblem} and defined by
\begin{equation}
\mathcal{F}(u):= \left\{ \begin{array}{ll} \displaystyle\int_\Omega |D_{x_1} u| + \int_{\partial\Omega_1 \times \Omega_2} |u| \, d\mathcal{H}^{N-1} + \sum_{i=2}^n \frac{1}{p_i} \int_\Omega |\nabla_{x_i} u|^{p_i} \, dx \quad &\hbox{if} \ u \in BV_{0,w}^{(p_i)}(\Omega) \cap L^2(\Omega); \\ \\ + \infty \quad &\hbox{if} \ u \in  L^2(\Omega) \setminus BV_{0,w}^{(p_i)}(\Omega).
\end{array}\right.
\end{equation}
The functional $\mathcal{F}$ is lower semicontinuous with respect to convergence in $L^2(\Omega)$. Without loss of generality, assume that $\mathcal{F}(u_n)$ converges to a finite limit. For $i = 2, ..., N$, up to taking a subsequence we have that $\nabla_{x_i} u_n \rightharpoonup g_i$ weakly in $L^{p_i}(\Omega; \mathbb{R}^{n_i})$; since $u_n$ converges in $L^2(\Omega)$ to $u$, we get that $g_i = \nabla_{x_i} u$. By the lower semicontinuity of the norms in $L^q(\Omega; \mathbb{R}^{n_i})$ we get lower semicontinuity of the term $\frac{1}{p_i} \int_\Omega |\nabla_{x_i} u|^{p_i} \, dx$. Finally, lower semicontinuity of the linear-growth term was already proved in Section \ref{sec:anisotropicspaces}.

Clearly, $\mathcal{F}$ is convex; since we proved that it is also lower semicontinuous with respect to convergence in $L^2(\Omega)$, by the Brezis-Komura theorem (Theorem \ref{thm:breziskomura}) there exists a unique strong solution of the abstract Cauchy problem
\begin{equation}\label{CPDirichlet}
\left\{ \begin{array}{ll} 0 \in u'(t) + \partial \mathcal{F}(u(t)) \quad \mbox{for } t \in [0,T]; \\[5pt] u(0) = u_0. \end{array}\right.
\end{equation}
To characterise the subdifferential of $\mathcal{F}$ in $L^2(\Omega)$, we will use convex duality in the setting presented in Section \ref{subsec:convexduality}; to this end, define the following multivalued operator.

\begin{definition}\label{dfn:operatorad}
{\rm We say that $(u,v) \in \mathcal{A}$ if and only if $u, v \in L^2(\Omega)$, $u \in BV_{0,w}^{(p_i)}(\Omega)$ and there exists a vector field $\z = (\z_1, ..., \z_n) \in X^{(p'_i)}_2(\Omega)$ such that the following conditions hold:
\begin{equation} 
\| \z_1 \|_\infty \leq 1;
\end{equation}
\begin{equation}
(\z_1, D_{x_1} u) = |D_{x_1} u| \quad \hbox{as measures};
\end{equation}
\begin{equation} 
\z_i = |\nabla_{x_i} u|^{p_i - 2} \nabla_{x_i} u \quad \mathcal{L}^N-\hbox{a.e. in } \Omega \hbox{ for all } i = 2,...,k;
\end{equation}
\begin{equation}
-\mathrm{div}(\z) = v \quad \hbox{in} \ \Omega;
\end{equation}
\begin{equation}
[\z_1, \nu_{x_1}] \in \mbox{sign}(-u) \qquad  \mathcal{H}^{N-1}-\mbox{a.e.  on} \ \partial \Omega_1 \times \Omega_2.
\end{equation}
}
\end{definition}

In the last condition, $\mbox{sign}$ denotes the multivalued sign function, i.e.
\begin{equation*}
\mbox{sign}(x) = \threepartdef{1}{\, x > 0;}{\, \mbox{[} -1,1 \mbox{]} }{\, x = 0;}{-1}{ \, x < 0.}
\end{equation*}

\begin{lemma}\label{lem:inclusion}
We have $\mathcal{A} \subset \partial \mathcal{F}$. In particular, $\mathcal{A}$ is a monotone operator.
\end{lemma}

\begin{proof}
Let $(u,v) \in \mathcal{A}$ and $\z \in X^{(p'_i)}_2(\Omega)$ satisfy the conditions in Definition \ref{dfn:operatorad}. Given $w \in L^2(\Omega) \cap BV_{0,w}^{(p_i)}(\Omega)$, we apply the anisotropic Gauss-Green formula (Theorem \ref{thm:anisotropicgaussgreen}) and get
\begin{align}
\int_\Omega (w - u) \, v \, dx &= - \int_\Omega \mathrm{div}(\z) (w - u) \, dx \\
&= \int_\Omega (\z_1, D_{x_1} w) + \sum_{i=2}^k \int_\Omega \z_i \cdot \nabla_{x_i} w \, dx - \int_{\partial\Omega_1 \times \Omega_2} w \, [\z_1, \nu_{x_1}] \,d\mathcal{H}^{N-1} \\
&\qquad\qquad - \int_\Omega (\z_1, D_{x_1} u) - \sum_{i=2}^k \int_\Omega \z_i \cdot \nabla_{x_i} u \, dx + \int_{\partial\Omega_1 \times \Omega_2} u \, [\z_1, \nu_{x_1}] \,d\mathcal{H}^{N-1} \\
&\leq \int_\Omega |D_{x_1} w| + \sum_{i=2}^k \int_\Omega |\nabla_{x_i} u|^{p_i-2} \nabla_{x_i} u \cdot \nabla_{x_i} w \, dx + \int_{\partial\Omega_1 \times \Omega_2} |w|  \,d\mathcal{H}^{N-1} \\
&\qquad\qquad - \int_\Omega |D_{x_1} u| - \sum_{i=2}^k \int_\Omega |\nabla_{x_i} u|^{p_i} \, dx - \int_{\partial\Omega_1 \times \Omega_2} |u|  \,d\mathcal{H}^{N-1} \\
&\leq \int_\Omega |D_{x_1} w| + \sum_{i=2}^k \int_\Omega \bigg( \frac{1}{p_i} |\nabla_{x_i} w|^{p_i} + \frac{p_i-1}{p_i} |\nabla_{x_i} u|^{p_i} \bigg) \, dx + \int_{\partial\Omega_1 \times \Omega_2} |w|  \,d\mathcal{H}^{N-1} \\
&\qquad\qquad - \int_\Omega |D_{x_1} u| - \sum_{i=2}^k \int_\Omega |\nabla_{x_i} u|^{p_i} \, dx - \int_{\partial\Omega_1 \times \Omega_2} |u|  \,d\mathcal{H}^{N-1}
\end{align}
\begin{align}
&= \int_\Omega |D_{x_1} w| + \sum_{i=2}^k \frac{1}{p_i} \int_\Omega |\nabla_{x_i} w|^{p_i} \, dx + \int_{\partial\Omega} |w| \, d\mathcal{H}^{N-1} \\
&\qquad\qquad\qquad -  \int_\Omega |D_{x_1} u| - \sum_{i=2}^k \frac{1}{p_i} \int_\Omega |\nabla_{x_i} u|^{p_i} \, dx - \int_{\partial\Omega} |u| \, d\mathcal{H}^{N-1} \\
&= \mathcal{F}(w) - \mathcal{F}(u),
\end{align}
which concludes the proof.
\end{proof}

Now, let us recall that a multivalued operator $\mathcal{A}$ is completely accretive if and only if the following condition is satisfied (see \cite{ACMBook,BCr2}):
\begin{equation}\label{eq:conditionforaccretivity}
\int_\Omega T(u^1 - u^2) (v^1 - v^2) \, dx \geq 0    
\end{equation}
for every $(u^1,v^1), (u^2,v^2) \in \mathcal{A}$ and all functions $T \in C^{\infty}(\mathbb{R})$ such that $0 \leq T' \leq 1$, $T'$ has compact support, and $x=0$ is not contained in the support of $T$.

\begin{lemma}\label{lem:completeaccretivity}
The operator $\mathcal{A}$ is completely accretive.
\end{lemma}

\begin{proof}
We need to show that condition \eqref{eq:conditionforaccretivity} holds for all $T \in C^{\infty}(\mathbb{R})$ satisfying the above conditions, i.e. such that $0 \leq T' \leq 1$, $T'$ has compact support, and $x=0$ is not contained in the support of $T$. For $j = 1,2$, let $(u^j, v^j) \in \mathcal{A}$ and let $\z^j$ be the associated vector fields. Observe that $T(u^1 - u^2) \in BV_{0,w}^{(p_i)}(\Omega)$. Since $\| \z^1_1 \|_\infty \leq 1$ and $\| \z^2_1 \|_\infty \leq 1$, by Proposition \ref{prop:propertiesofanzellottipairing} for every Borel set $B \subset \Omega$ we have
\begin{align}
\int_{B} (\z^1_1 - \z^2_1, &D_{x_1}(u^1 - u^2)) \\
&=  \int_{B} \vert D_{x_1} u^1 \vert - \int_{B} (\z^1_1, D_{x_1} u^2) +  \int_{B} \vert D_{x_1} u^2 \vert - \int_{B} (\z^2_1, D_{x_1} u^1) \geq 0.
\end{align}
By definition of the Radon-Nikodym derivative $\theta(\z^1_1 - \z^2_1, D_{x_1} (u^1 - u^2), x)$ we get
$$ \int_{B} \theta(\z^1_1 - \z^2_1, D_{x_1} (u^1 - u^2), x) \, d \vert D_{x_1}(u^1 - u^2) \vert = \int_{B} (\z^1_1 - \z^2_1, D_{x_1} (u^1 - u^2)) \geq 0$$
for all Borel sets $B \subset \Omega$. Therefore,
$$\theta(\z^1_1 - \z^2_1, D_{x_1} (u^1 - u^2), x) \geq 0 \qquad \vert D_{x_1}(u^1 - u^2) \vert-\hbox{a.e. on } \Omega$$
and since $|D_{x_1} T(u^1 - u^2)|$ is absolutely continuous with respect to $|D_{x_1}(u^1 - u^2)|$, we also have
\begin{equation}\label{eq:estimateontheta}
\theta(\z^1_1 - \z^2_1, D_{x_1}(u^1 - u^2), x) \geq 0 \qquad \vert D_{x_1} T(u^1 - u^2) \vert-\hbox{a.e. on } \Omega.
\end{equation}
By Proposition \ref{prop:lipschitzfunctionofdensity}, we get that
\begin{equation}
\theta(\z^1_1 - \z^2_1, D_{x_1} T(u^1 - u^2), x) \geq 0 \qquad \vert D_{x_1} T(u^1 - u^2) \vert-\hbox{a.e. on } \Omega,
\end{equation}
so
\begin{align}\label{eq:completeaccretivityfirstestimate}
\int_{\Omega} (\z^1_1 - \z^2_1, D_{x_1} &T(u^1 - u^2)) \\
&=  \int_{\Omega} \theta(\z^1_1 - \z^2_1, D_{x_1} T(u^1 - u^2), x) \, d \vert D_{x_1} T(u^1 - u^2) \vert \geq 0.
\end{align}

Moreover, for $i = 2,...,k$ we have\begin{align}\label{eq:completeaccretivitymiddleestimate}
\int_{\Omega} \nabla_{x_i} &T(u^1-u^2) \cdot (\z^1_i - \z^2_i) \, dx  =  \int_{\Omega} T'(u^1-u^2) \, \nabla_{x_i} (u^1-u^2) \cdot (\z^1_i - \z^2_i) \, dx \\
&= \int_{\Omega} T'(u^1-u^2) \, \nabla_{x_i} u^1 \cdot \z^1_i \, dx -  \int_{\Omega} T'(u^1-u^2) \, \nabla_{x_i} u^1 \cdot \z^2_i \, dx \\
&\qquad\qquad - \int_{\Omega} T'(u^1-u^2) \, \nabla_{x_i} u^2 \cdot \z^1_i \, dx + \int_{\Omega} T'(u^1-u^2) \, \nabla_{x_i} u^2 \cdot \z^2_i \, dx \\
&= \int_{\Omega} T'(u^1-u^2) \, |\nabla_{x_i} u^1|^{p_i} \, dx -  \int_{\Omega} T'(u^1-u^2) \, \nabla_{x_i} u^1 \cdot \z^2_i \, dx \\
&\qquad\qquad - \int_{\Omega} T'(u^1-u^2) \, \nabla_{x_i} u^2 \cdot \z^1_i \, dx + \int_{\Omega} T'(u^1-u^2) \, |\nabla_{x_i} u^2|^{p_i} \, dx \\
&\geq \int_{\Omega} T'(u^1-u^2) \, |\nabla_{x_i} u^1|^{p_i} \, dx - \frac{1}{p_i} \int_{\Omega} T'(u^1-u^2) \, |\nabla_{x_i} u^1|^{p_i} \, dx \\
&\qquad\qquad  - \frac{1}{p'_i} \int_{\Omega} T'(u^1-u^2) \, \vert \z^2_i \vert^{p'_i} \, dx - \frac{1}{p_i} \int_{\Omega} T'(u^1-u^2) \, |\nabla_{x_i} u^2|^{p_i} \, dx \\
&\qquad\qquad - \frac{1}{p'_i} \int_{\Omega} T'(u^1-u^2) \, \vert \z^1_i \vert^{p'_i} \, dx + \int_{\Omega} T'(u^1-u^2) \, |\nabla_{x_i} u^2|^{p_i} \, dx \\
&\geq \frac{1}{p'_i} \int_{\Omega} T'(u^1-u^2) \, |\nabla_{x_i} u^1|^{p_i} \, dx - \frac{1}{p'_i} \int_{\Omega} T'(u^1-u^2) \, \vert \z^1_i \vert^{p'_i} \, dx \\
&\qquad\qquad + \frac{1}{p'_i} \int_{\Omega} T'(u^1-u^2) \, |\nabla_{x_i} u^2|^{p_i} \, dx - \frac{1}{p'_i} \int_{\Omega} T'(u^1-u^2) \, \vert \z^2_i \vert^{p'_i} \, dx = 0,
\end{align}
since $T' \geq 0$, $\z^1_i = |\nabla_{x_i} u^1|^{p_i-2} \nabla_{x_i} u^1$ a.e. in $\Omega$ and $\z^2_i = |\nabla_{x_i} u^2|^{p_i-2} \nabla_{x_i} u^2$ a.e. in $\Omega$.

As a final preparation, we will prove that
\begin{equation}\label{eq:completeaccretivityfinalestimate}
\int_{\partial\Omega_1 \times \Omega_2} - T(u^1 - u^2) [\z_1^1 - \z_1^2, \nu_{x_1}] \, d\mathcal{H}^{N-1} \geq 0.
\end{equation}
To do this, let us consider several cases depending on the signs of $u^1$ and $u^2$ at a given point $x \in \partial\Omega_1 \times \Omega_2$.

\begin{enumerate}
\item $u^1(x), u^2(x) > 0$: then, $[\z_1^1, \nu_{x_1}](x) = [\z_1^2, \nu_{x_1}](x) = -1$, so the integrand in \eqref{eq:completeaccretivityfinalestimate} equals zero. A similar argument works if $u^1(x), u^2(x) < 0$.

\item $u^1(x) > 0 > u^2(x)$: then, $[\z_1^1, \nu_{x_1}](x) = -1$ and $[\z_1^2, \nu_{x_1}](x) = 1$, so $[\z_1^1 - \z_1^2, \nu_{x_1}](x) < 0$. By our assumptions on $T$, we have that $T(u^1(x) - u^2(x)) \geq 0$, so the integrand in \eqref{eq:completeaccretivityfinalestimate} is nonnegative. A similar argument works whenever $u^1(x) < 0 < u^2(x)$.

\item $u^1(x) > 0 = u^2(x)$: then, $[\z_1^1, \nu_{x_1}](x) = -1$ and $[\z_1^2, \nu_{x_1}](x) \in [-1, 1]$, so $[\z_1^1 - \z_1^2, \nu_{x_1}](x) \leq 0$. By our assumptions on $T$, we have that $T(u^1(x) - u^2(x)) = T(u^1(x)) \geq 0$, so the integrand in \eqref{eq:completeaccretivityfinalestimate} is nonnegative. A similar argument works whenever $u^1(x) < 0 = u^2(x)$.

\item $u^1(x) = 0 > u^2(x)$: then, $[\z_1^1, \nu_{x_1}](x) \in [-1,1]$ and $[\z_1^2, \nu_{x_1}](x) = 1$, so $[\z_1^1 - \z_1^2, \nu_{x_1}](x) \leq 0$. By our assumptions on $T$, we have that $T(u^1(x) - u^2(x)) = T(-u^2(x)) \geq 0$, so the integrand in \eqref{eq:completeaccretivityfinalestimate} is nonnegative. A similar argument works whenever $u^1(x) = 0 < u^2(x)$.

\item $u^1(x) = u^2(x) = 0$: then, $T(u^1(x) - u^2(x)) = T(0) = 0$, so the integrand in \eqref{eq:completeaccretivityfinalestimate} equals zero. We covered all the cases depending on the signs of $u^1$ and $u^2$, so the proof of the claim \eqref{eq:completeaccretivityfinalestimate} is concluded.
\end{enumerate}

To conclude that the operator $\mathcal{A}$ is completely accretive, we now apply the anisotropic Gauss-Green formula (Theorem \ref{thm:anisotropicgaussgreen}) and collect the estimates \eqref{eq:completeaccretivityfirstestimate}, \eqref{eq:completeaccretivitymiddleestimate} and \eqref{eq:completeaccretivityfinalestimate} to get
\begin{align}
\int_{\Omega}T(u^1-u^2) (v^1-v^2)\, dx = - \int_{\Omega} T(u^1 &- u^2) (\mbox{div}(\z^1)- \mbox{div}(\z^2)) \, dx \\
= \int_{\Omega} (\z^1_1 - \z^2_1, D_{x_1} T(u^1 - u^2)) &+ \sum_{i=2}^k \int_{\Omega} \nabla_{x_i} T(u^1-u^2) \cdot (\z^1_i - \z^2_i) \, dx \\
&- \int_{\partial\Omega_1 \times \Omega_2} T(u^1 - u^2) [\z_1^1 - \z_1^2, \nu_{x_1}] \, d\mathcal{H}^{N-1} \geq 0,
\end{align}
so $\mathcal{A}$ satisfies the condition \eqref{eq:conditionforaccretivity} and thus is completely accretive.
\end{proof}

We now prove the anticipated result that we can characterise the subdifferential of $\mathcal{F}$ using the auxiliary operator $\mathcal{A}$.

\begin{theorem}\label{thm:subdifferentialgradientflow}
We have
$$ \mathcal{A} = \partial \mathcal{F}$$
and $D(\mathcal{A})$ is dense in $L^2(\Omega)$.
\end{theorem}

\begin{proof} 
{\bf Step 1.} By Lemma \ref{lem:inclusion}, the operator $\mathcal{A}$ is monotone and contained in $\partial \mathcal{F}$. The operator $\partial \mathcal{F}$ is maximal monotone; hence, once we prove that $\mathcal{A}$ satisfies the range condition, i.e.
\begin{equation}\label{eq:rangecondition}
\forall \, g \in  L^2(\Omega) \ \exists \, u \in D(\mathcal{A}) \,\, \mbox{such that} \,\,  g \in u + \mathcal{A}(u),
\end{equation}
or equivalently that $(u, g-u) \in \mathcal{A}$, the Minty theorem implies that that the operator $\mathcal{A}$ is maximal monotone and consequently that $\mathcal{A} = \partial \mathcal{F}$. Therefore, we need to prove existence of $u \in BV_{0,w}^{(p_i)}(\Omega)$ and $\z = (\z_1, ..., \z_n) \in X^{(p'_i)}_2(\Omega)$ such that the following conditions hold:
\begin{equation}\label{eq:linftybound}
\| \z_1 \|_\infty \leq 1;
\end{equation}
\begin{equation}\label{eq:firstcoordinate}
(\z_1, D_{x_1} u) = |D_{x_1} u| \quad \hbox{as measures};
\end{equation}
\begin{equation}\label{eq:othercoordinates}
\z_i = |\nabla_{x_i} u|^{p_i - 2} \nabla_{x_i} u \quad \mathcal{L}^N-\hbox{a.e. in } \Omega \hbox{ for all } i = 2,...,k;
\end{equation}
\begin{equation}\label{eq:divergenceconstraint}
-\mathrm{div}(\z) = g-u \quad \hbox{in} \ \Omega;
\end{equation}
\begin{equation}\label{eq:boundarycondition}
[\z_1, \nu_{x_1}] \in \mbox{sign}(-u) \qquad  \mathcal{H}^{N-1}-\mbox{a.e.  on} \ \partial \Omega_1 \times \Omega_2.
\end{equation}
We will prove that the range condition \eqref{eq:rangecondition} holds using the Fenchel-Rockafellar duality theorem; we need to present it in the framework described before Theorem \ref{thm:fenchelrockafellar}. 

{\flushleft \bf Step 2.} We first restrict our attention to $W_{0,w}^{1,(p_i)}(\Omega)$ and set 
$$U = W_{0,w}^{1,(p_i)}(\Omega) \cap L^2(\Omega)$$ 
and 
$$V = L^1(\partial\Omega_1 \times \Omega_2) \times L^1(\Omega; \mathbb{R}^{n_1}) \times ... \times L^{p_n}(\Omega; \mathbb{R}^{n_k}). $$
The operator $A: U \rightarrow V$ is defined by the formula
$$ Au = (u|_{\partial\Omega_1 \times \Omega_2}, \nabla_{x_1} u, ..., \nabla_{x_k} u).$$
Clearly, $A$ is a linear operator. Let us prove that it is also continuous. For the zeroth coordinate, it is a composition of three linear and continuous maps: an embedding of the space $W_{0,w}^{1,(p_i)}(\Omega)$ into $W^{1,1}(\Omega)$; the trace operator from $W^{1,1}(\Omega)$ to $L^1(\partial\Omega)$; and a restriction of a function on $\partial\Omega$ to $\partial\Omega_1 \times \Omega_2$. For the other coordinates, it is enough to note that $W_{0,w}^{1,(p_i)}(\Omega)$ is a subspace of $W^{1,(p_i)}(\Omega)$.

We denote the points $v \in V$ in the following way: $v = (v_0, v_1, ..., v_k)$, where $v_0 \in L^1(\partial\Omega_1 \times \Omega_2)$ and $v_i \in L^{p_i}(\Omega;\mathbb{R}^{n_i})$ for $i = 1,...,k$. We will also need the explicit expression of the dual space to $V$, which is
$$ V^* =  L^\infty(\partial\Omega_1 \times \Omega_2) \times L^\infty(\Omega;\mathbb{R}^{n_1}) \times ... \times L^{p'_k}(\Omega; \R^{n_k}),$$
and we use a similar notation for points $v^* \in V^*$. Then, we set $E: V \rightarrow \mathbb{R}$ by the formula
\begin{equation}
E(v_0,v_1,...,v_k) = E_0(v_0) + \sum_{i=1}^k E_i(v_i),
\end{equation}
where
\begin{equation}
E_0(v_0) = \int_{\partial\Omega_1 \times \Omega_2} \vert v_0 \vert \,d\mathcal{H}^{N-1}
\end{equation}
and
\begin{equation}
E_i(v_i) = \frac{1}{p_i} \int_\Omega |v_i|^{p_i} \, dx.
\end{equation}
Clearly, $E$ is a proper and convex functional, and by checking coordinate-wise we see that it is also lower semicontinuous. We also set $G: W_{0,w}^{1,(p_i)}(\Omega) \cap L^2(\Omega) \rightarrow \mathbb{R}$ by
$$G(u):= \frac12 \int_\Omega u^2 \, dx - \int_\Omega ug \, dx$$
and see that it a proper, convex, and continuous functional. Observe that by the Young inequality
\begin{equation}
G(u) \geq \frac12 \int_\Omega u^2 \, dx - \varepsilon \int_\Omega u^2 \, dx - C(\varepsilon) \int_\Omega g^2 \, dx,
\end{equation}
so if we choose $\varepsilon < \frac12$, we get that $G$ is bounded from below.

{\flushleft \bf Step 3.} We now compute the convex conjugates of $E$ and $G$. The dual functional of $G$, i.e. $G^* : (W_{0,w}^{1,(p_i)}(\Omega) \cap L^2(\Omega))^* \rightarrow [0,+\infty ]$ is given by
$$G^*(u^*) = \displaystyle\frac12 \int_\Omega (u^*  + g)^2 \, dx.$$
Now, observe that we evaluate $G^*$ at $A^* v^*$; we need to compute this value. By definition of the dual operator, we get
\begin{align}
\int_\Omega u \, (A^* v^*) \, dx &= \langle u, A^* v^* \rangle  = \langle v^*, Au \rangle = \int_{\partial\Omega_1 \times \Omega_2} v^*_0 \, u \, d\mathcal{H}^{N-1} + \sum_{i=1}^k \int_\Omega v_i^* \cdot \nabla_{x_i} u \, dx.
\end{align}
If we now consider only functions $u \in W_0^{1,(p_i)}(\Omega) \cap L^2(\Omega)$, which are dense in $L^2(\Omega)$, we see that the boundary term disappears and get
\begin{equation}\label{eq:neumanndiv0}
A^* v^* =- \mathrm{div}((v_1^*, ..., v_k^*)).
\end{equation}
In particular, the divergence of $(v_1^*, ..., v_k^*)$ is square-integrable, so $(v_1^*, ..., v_k^*) \in X_2^{(p'_i)}(\Omega)$. Therefore, for any $u \in W_{0,w}^{1,(p_i)}(\Omega) \cap L^2(\Omega)$ we may apply the anisotropic Gauss-Green formula (Theorem \ref{thm:anisotropicgaussgreen}) and get
\begin{align}
\int_\Omega u \, (A^* v^*) \, dx &= \langle u, A^* v^* \rangle  = \langle v^*, Au \rangle = \int_{\partial\Omega_1 \times \Omega_2} v^*_0 \, u \, d\mathcal{H}^{N-1} + \sum_{i=1}^k \int_\Omega v_i^* \cdot \nabla_{x_i} u \, dx \\
& = \int_{\partial\Omega_1 \times \Omega_2} v^*_0 \, u \, d\mathcal{H}^{N-1} - \int_\Omega u \, \mathrm{div}((v_1^*,...,v_k^*)) \, dx + \int_{\partial\Omega_1 \times \Omega_2} u \, [v_1^*, \nu_{x_1}] \, d\mathcal{H}^{N-1} \\
&= - \int_\Omega u \, \mathrm{div}((v_1^*,...,v_k^*)) \, dx + \int_{\partial\Omega_1 \times \Omega_2} u \, (v_0^*  + [v_1^*, \nu_{x_1}]) \, d\mathcal{H}^{N-1}.
\end{align}
By \eqref{eq:neumanndiv0}, the integrals over $\Omega$ cancel out,  so
\begin{equation}\label{eq:pandpoagreepart1}
\int_{\partial\Omega_1 \times \Omega_2} u \, (v_0^* + [v_1^*, \nu_{x_1}]) \, d\mathcal{H}^{N-1} = 0
\end{equation}
for all $u \in W_{0,w}^{1,(p_i)}(\Omega) \cap L^2(\Omega)$. 

To conclude that
\begin{equation}\label{eq:pandpoagree}
v_0^* = - [v_1^*, \nu_{x_1}] \qquad \mathcal{H}^{N-1}-\mbox{a.e. on } \partial\Omega_1 \times \Omega_2,
\end{equation}
we need to show that the set of traces of functions in $W_{0,w}^{1,(p_i)}(\Omega) \cap L^2(\Omega)$ is a dense subset of $L^1(\partial\Omega_1 \times \Omega_2)$. To see this, take any $v \in L^1(\partial\Omega_1 \times \Omega_2)$, fix $\varepsilon > 0$ and take a Lipschitz function $h \in \mathrm{Lip}(\partial\Omega_1 \times \Omega_2)$ with compact support in $\partial\Omega_1 \times \Omega_2$ (in the topology relative to $\partial\Omega$) such that $\| v - h \|_1 \leq \varepsilon$. By the McShane extension theorem, we can find a function $H \in \mathrm{Lip}(\overline{\Omega})$ such that $H|_{\partial\Omega_1 \times \Omega_2} = h$. Now, let
\begin{equation}
\delta = \mathrm{dist}(\mathrm{supp} \, h, \overline{\Omega_1} \times \partial\Omega_2).
\end{equation}
Observe that $\delta > 0$ since $h$ has compact support and that
\begin{equation}
d_\delta(x) = \twopartdef{\frac{1}{\delta} \mathrm{dist}(x, \overline{\Omega_1} \times \partial\Omega_2)}{\mbox{if } \mathrm{dist}(x, \overline{\Omega_1} \times \partial\Omega_2) < \delta;}{1}{\mbox{otherwise}}
\end{equation}
is a bounded Lipschitz function. Then, $u = H d_\delta$ is a bounded Lipschitz function as a product of two such functions, so it lies in $W^{1,(p_i)}(\Omega) \cap L^2(\Omega)$. Clearly, we have that $u = 0$ on $\Omega_1 \times \partial\Omega_2$, so actually $u \in W_{0,w}^{1,(p_i)}(\Omega) \cap L^2(\Omega)$. Since $u = h$ on $\partial\Omega_1 \times \Omega_2$, the set of traces of functions in $W_{0,w}^{1,(p_i)}(\Omega) \cap L^2(\Omega)$ is dense in $L^1(\partial\Omega_1 \times \Omega_2)$. Therefore, equation \eqref{eq:pandpoagreepart1} implies that
$$ \int_{\partial\Omega_1 \times \Omega_2} w \, (v_0^*  + [v_1^*, \nu_{x_1}]) \, d\mathcal{H}^{N-1} = 0$$
for $w$ in a dense subset of $L^1(\partial\Omega_1 \times \Omega_2)$, which concludes the proof of \eqref{eq:pandpoagree}.

Therefore,
\begin{equation}\label{eq:formulaforgstar}
G^*(A^*v^*) = \displaystyle\frac12 \int_\Omega ( -\mathrm{div}((v_1^*,...,v_k^*))  + g)^2 \, dx.
\end{equation}
We now turn to computing the convex conjugates of the functionals $E_i$ (for $i = 1,...,k$). To compute the functional $E_0^*: L^\infty(\partial\Omega_1 \times \Omega_2) \rightarrow \mathbb{R} \cup \{ \infty \}$,
observe that
$$E_0^*(v_0^*) = \sup_{v_0 \in L^1(\partial\Omega_1 \times \Omega_2)} \left\{ \int_{\partial \Omega_1 \times \Omega_2} v_0 \, v_0^* \, d\mathcal{H}^{N-1} - \int_{\partial\Omega_1 \times \Omega_2} |v_0| \, d\mathcal{H}^{N-1} \right\}.$$
Therefore,
\begin{equation}\label{eq:formulafore0star}
E_0^*(v_0^*) = \left\{ \begin{array}{ll} \displaystyle  0 \quad &\hbox{if} \ \  |v_0^*| \leq 1 \quad \mathcal{H}^{N-1}-\mbox{a.e. on } \partial\Omega_1 \times \Omega_2; \\[10pt] +\infty \quad &\hbox{otherwise.}   \end{array}  \right.
\end{equation}
Using \cite[Proposition IV.1.2]{EkelandTemam}, we see that the functional $E_1^*: L^\infty(\Omega; \mathbb{R}^{n_1}) \rightarrow [0,\infty]$ is given by the formula
$$E_1^*(v_1^*) = \twopartdef{0}{\mbox{if } \| v_1^* \|_\infty \leq 1;}{+\infty}{\mbox{otherwise}} $$
and for $i = 2, ..., k$ the functional $E_i^*: L^{p'_i}(\Omega; \mathbb{R}^{n_i}) \rightarrow [0,\infty]$ is given by
$$ E_i^*(v_i^*) = \frac{1}{p'_i} \int_\Omega |v_i^*|^{p'_i} \, dx, $$
so we computed the convex conjugate of $E$ coordinate-wise.

{\flushleft \bf Step 4.} We will infer that the range condition \eqref{eq:rangecondition} holds in the following way. Consider the minimisation problem
\begin{equation}\label{eq:mimimisationofeg}
\inf_{u \in U} \bigg\{ E(Au) + G(u) \bigg\}
\end{equation}
with $E$ and $G$ defined as above. For $u_0 \equiv 0$ we have $E(Au_0) = G(u_0) = 0 < \infty$ and $E$ is continuous at $0$. Then, Theorem \ref{thm:fenchelrockafellar} implies that the dual problem given by
\begin{equation}
\sup_{v^* \in V^*} \bigg\{ - E^*(-v^*) - G^*(A^* v^*) \bigg\}
\end{equation}
admits at least one solution and there is no duality gap, i.e. the infimum in the first problem is equal to the supremum in the second one. We first split the dual problem into coordinates and re-write it as
\begin{equation}
\sup_{v^* \in L^\infty(\partial\Omega_1 \times \Omega_2) \times L^\infty(\Omega; \mathbb{R}^{n_1}) \times ... \times L^{p'_k}(\Omega; \mathbb{R}^{n_k})} \bigg\{ - E_0^*(-v_0^*) - \sum_{i=1}^k E_i^*(-v_i^*) - G^*(A^* v^*) \bigg\}.
\end{equation}
Recall that in order for $G^*(A^* v^*)$ to be finite, we necessarily have that $(v_1^*, ..., v_k^*) \in X_2^{(p'_i)}(\Omega)$. Moreover, observe that $E_0^*$ and $E_1^*$ take only values $0$ and $+\infty$. Therefore, from now on we will denote by $\mathcal{Z}$ the subset of $V^*$ such that the dual problem does not immediately return $-\infty$. To be exact,
\begin{align}
\mathcal{Z} = \bigg\{ v^* \in L^\infty(\partial\Omega_1 \times \Omega_2) &\times X_2^{(p'_i)}(\Omega): \quad |v_1^*| \leq 1 \quad \mathcal{L}^N-\mbox{a.e in } \Omega; \\
& |v_0^*| \leq 1 \quad \mathcal{H}^{N-1}-\mbox{a.e. on } \partial\Omega_1 \times \Omega_2; \quad v_0^* = -[v_1^*, \nu_{x_1}] \bigg\}.
\end{align}
Using the newly defined set $\mathcal{Z}$ and the formulas for $E_0^*$ and $E_1^*$, we can rewrite the dual problem as
\begin{equation}
\sup_{v^* \in \mathcal{Z}} \bigg\{ - \sum_{i = 2}^k E_i^*(-v_i^*) - G^*(A^* v^*) \bigg\},
\end{equation}
and using the explicit expressions for $E_i^*$ (for $i = 2,...,k$) and $G^*$ we finally rewrite it as
\begin{equation}
\sup_{v^* \in \mathcal{Z}} \bigg\{ - \sum_{i = 2}^k \frac{1}{p'_i} \int_\Omega |v_i^*|^{p'_i} \, dx - \frac12 \int_\Omega ( -\mathrm{div}((v_1^*,...,v_k^*))  + g)^2 \, dx \bigg\}.
\end{equation}
Note that due to the constraint $v_0^* = - [v_1^*, \nu_{x_1}]$, the first coordinate of any solution to the dual problem is determined by the other coordinates (the constraint that $|v_0^*| \leq 1$ $\mathcal{H}^{N-1}$-a.e. on $\partial\Omega_1 \times \Omega_2$ follows from the $L^\infty$ bound on $v_1^*$ and Theorem \ref{thm:definitionofweaktrace}).


{\flushleft \bf Step 5.} Now, consider the functional $\mathcal{G} : L^2(\Omega) \rightarrow (-\infty, + \infty]$ defined by
\begin{equation}
\mathcal{G} (v) := \mathcal{F}(v) + G(v),
\end{equation}
i.e. an extension of the functional $E \circ A + G$, well-defined for functions in $W_{0,w}^{1,(p_i)}(\Omega) \cap L^2(\Omega)$, to the space $BV_{0,w}^{(p_i)}(\Omega) \cap L^2(\Omega)$ (and a further extension by $+\infty$ to the rest of $L^2(\Omega)$). By the properties of $\mathcal{F}$ and $G$, we get that $\mathcal{G}$ is bounded from below, convex and lower semicontinuous. It is also coercive, because whenever $\mathcal{G}(u) \leq M$, we have
\begin{equation}
\frac12 \int_\Omega u^2 \, dx + \mathcal{F}(u) \leq M + \int_\Omega u g \, dx,
\end{equation}
and by positivity of $\mathcal{F}$ and the Young inequality for $\varepsilon < \frac12$ we get
\begin{equation}
\bigg(\frac12 - \varepsilon \bigg) \int_\Omega u^2 \, dx \leq M + C(\varepsilon) \int_\Omega g^2 \, dx,
\end{equation}
so the norm of $u$ in $L^2(\Omega)$ is bounded. Therefore, the minimisation of $\mathcal{G}(v)$ in $L^2(\Omega)$ admits a solution $u$ and by the anisotropic Meyers-Serrin theorem (Proposition \ref{prop:meyersserrin}) we have
\begin{equation}
\min_{v \in L^2(\Omega)} \mathcal{G}(v) = \inf_{v \in U} \bigg\{ E(Av) + G(v) \bigg\}.
\end{equation}
However, the solution $u$ does not necessarily lie in $W_{0,w}^{1,(p_i)}(\Omega)$, which is the domain of the functional $E \circ A + G$. Therefore, we cannot use the extremality conditions given in Theorem \ref{thm:fenchelrockafellar}, and we instead rely on the $\varepsilon-$subdifferentiability property of minimising sequences given in \eqref{eq:epsilonsubdiff1} and \eqref{eq:epsilonsubdiff2}. From this, we will deduce that the vector field $\mathbf{z} = -(v_1^*,...,v_k^*) \in X_2^{(p'_i)}(\Omega)$ satisfies the conditions \eqref{eq:linftybound}-\eqref{eq:boundarycondition} required for the range condition \eqref{eq:rangecondition}. Observe that condition \eqref{eq:linftybound} is automatically satisfied; we proceed to prove the other conditions.

Take a sequence $u_n \in W_{0,w}^{1,(p_i)}(\Omega) \cap L^2(\Omega)$ which approximates $u$ as in the anisotropic Meyers-Serrin theorem (Proposition \ref{prop:meyersserrin}); in particular, it is a minimising sequence in \eqref{eq:mimimisationofeg}. By the second subdifferentiability property \eqref{eq:epsilonsubdiff2}, for every $w \in L^2(\Omega)$ we have
$$G(w) - G(u_n) \geq \langle (w -u_n), A^* v^* \rangle  -\varepsilon_n,$$
and by passing to the limit $n \rightarrow \infty$ we get
$$G(w) - G(u) \geq \langle (w -u), A^* v^* \rangle.$$
Therefore,
$$ \mathrm{div}(-(v_1^*,...,v_k^*)) = A^* v^* \in \partial G(u) = \{ u - g \}, $$
so the divergence constraint \eqref{eq:divergenceconstraint} is satisfied once we choose $\mathbf{z} = - (v_1^*,...,v_k^*)$.

By the first subdifferentiability property \eqref{eq:epsilonsubdiff1}, we have
\begin{align}
0 \leq \int_{\partial\Omega_1 \times \Omega_2} |u_n| \, d\mathcal{H}^{N-1} &+ \sum_{i=1}^k \frac{1}{p_i} \int_\Omega |\nabla_{x_i} u_n|^{p_i} \, dx + \sum_{i=2}^k \frac{1}{p'_i} \int_\Omega |v_i^*|^{p'_i} \, dx \\
&+ \int_{\partial\Omega_1 \times \Omega_2} v_0^* u_n \, d\mathcal{H}^{N-1} + \sum_{i=1}^k \int_\Omega v_i^* \cdot \nabla_{x_i} u_n \, dx  \leq \varepsilon_n.
\end{align}
Since we assumed that the trace of $u_n$ on $\partial\Omega_1 \times \Omega_2$ is fixed and equal to the trace of $u$, the boundary integral does not change with $n$, so it is equal to zero. Therefore,
\begin{equation}
v_0^* \in \mathrm{sign}(-u) \qquad  \mathcal{H}^{N-1}-\mbox{a.e.  on} \ \partial \Omega_1 \times \Omega_2,
\end{equation}
and due to the constraint $v_0^* = [-v_1^*, \nu_{x_1}]$ we obtain that condition \eqref{eq:boundarycondition} holds for $\mathbf{z} = -(v_1^*,...,v_k^*)$.

Once we take into account that the boundary terms disappear, the first subdifferentiability property \eqref{eq:epsilonsubdiff1} yields
\begin{align}\label{eq:beforepassingtothelimit}
0 \leq \sum_{i=1}^k \frac{1}{p_i} \int_\Omega |\nabla_{x_i} u_n|^{p_i} \, dx + \sum_{i=2}^k \frac{1}{p'_i} \int_\Omega |v_i^*|^{p'_i} \, dx + \sum_{i=1}^k \int_\Omega v_i^* \cdot \nabla_{x_i} u_n \, dx  \leq \varepsilon_n.
\end{align}
Since $u_n|_{\partial\Omega_1 \times \Omega_2} = u|_{\partial\Omega_1 \times \Omega_2}$, by the anisotropic Green formula (Theorem \ref{thm:anisotropicgaussgreen}) we have
\begin{align}
\sum_{i=1}^k \int_\Omega v_i^* \cdot \nabla_{x_i} u_n \, dx &= - \int_{\Omega} u_n \, \mathrm{div}((v_1^*,...,v_k^*)) \, dx + \int_{\partial\Omega_1 \times \Omega_2} u_n \, [v_1^*, \nu_{x_1}] \,d\mathcal{H}^{N-1} \\
&= - \int_{\Omega} u \, \mathrm{div}((v_1^*,...,v_k^*)) \, dx + \int_{\partial\Omega_1 \times \Omega_2} u \, [v_1^*, \nu_{x_1}] \, d\mathcal{H}^{N-1} \\
&\qquad\qquad\qquad\qquad\qquad\qquad\, + \int_{\Omega} (u - u_n) \, \mathrm{div}((v_1^*,...,v_k^*)) \, dx \\
&= \int_{\Omega} (v_1^*, D_{x_1} u) + \sum_{i=2}^k \int_\Omega v_i^* \cdot \nabla_{x_i} u \, dx + \int_{\Omega} (u - u_n) \, \mathrm{div}((v_1^*,...,v_k^*)) \, dx.
\end{align}
Passing to the limit $n \rightarrow \infty$, we get
\begin{align}
\lim_{n \rightarrow \infty} \sum_{i=1}^k \int_\Omega v_i^* \cdot \nabla_{x_i} u_n \, dx = \int_{\Omega} (v_1^*, D_{x_1} u) + \sum_{i=2}^k \int_\Omega v_i^* \cdot \nabla_{x_i} u \, dx.
\end{align}
We now pass to the limit $n \rightarrow \infty$ in the inequality \eqref{eq:beforepassingtothelimit} and obtain
\begin{align}
\int_\Omega |D_{x_1} u| + \int_{\Omega} (v_1^*, D_{x_1} u)  + \sum_{i=2}^k \int_\Omega \bigg( \frac{1}{p_i} |\nabla_{x_i} u|^{p_i} + \frac{1}{p'_i} |v_i^*|^{p'_i} + v_i^* \cdot \nabla_{x_i} u \bigg) dx = 0.
\end{align}
However, on each coordinate the corresponding term is nonnegative; since $\| v_1^* \|_\infty \leq 1$, by Proposition \ref{prop:propertiesofanzellottipairing} we have
\begin{align}
\int_\Omega |D_{x_1} u| + \int_{\Omega} (v_1^*, D_{x_1} u) \geq 0
\end{align}
and for each $i = 2,...,k$ we have a pointwise inequality
\begin{align}
\frac{1}{p_i} |\nabla_{x_i} u|^{p_i} + \frac{1}{p'_i} |v_i^*|^{p'_i} + v_i^* \cdot \nabla_{x_i} u \geq 0
\end{align}
almost everywhere in $\Omega$. Therefore, the two inequalities above need to be equalities, so properties \eqref{eq:firstcoordinate} and \eqref{eq:othercoordinates} respectively hold for the choice $\z = -(v_1^*,...,v_k^*)$. Therefore, we proved that all the conditions \eqref{eq:linftybound}-\eqref{eq:boundarycondition} needed for the range condition \eqref{eq:rangecondition} hold, so the operator $\mathcal{A}$ is maximal monotone. Finally, by Proposition \ref{prop:domain} we have
$$ D(\partial \mathcal{F}) \subset  D(\mathcal{F}) =  BV_{0,w}^{(p_i)}(\Omega) \cap L^2(\Omega) \subset \overline{D(\mathcal{F})}^{L^2(\Omega)} \subset \overline{D(\partial \mathcal{F})}^{L^2(\Omega)},$$
and consequently the domain of $\partial \mathcal{F}$ is dense in $L^2(\Omega)$.
\end{proof}

In light of the previous Theorem, we can give the following definition of solutions to the homogeneous Dirichlet problem \eqref{eq:mainproblem}.

\begin{definition}\label{def:dirichlet1p}
{\rm Given $u_0 \in L^2(\Omega)$, we say that $u$ is a {\it weak solution} to the Dirichlet problem \eqref{eq:mainproblem} in $[0,T]$, if $u \in C([0,T]; L^2(\Omega)) \cap W_{\rm loc}^{1,2}(0, T; L^2(\Omega))$,   $u(0, \cdot) =u_0$, and for almost all $t \in (0,T)$
\begin{equation}\label{def:dirichletflow}
0 \in u_t(t, \cdot) + \mathcal{A} u(t, \cdot).
\end{equation}
In other words, for almost all $t \in (0,T)$ we have $u(t) \in BV_{0,w}^{(p_i)}(\Omega)$ and there exist vector fields $\z(t) \in X_2^{(p'_i)}(\Omega)$ such that the following conditions hold:
\begin{equation}
\| \z_1(t) \|_\infty \leq 1;
\end{equation}
\begin{equation}
(\z_1(t), D_{x_1} u(t)) = |D_{x_1} u(t)| \quad \hbox{as measures};
\end{equation}
\begin{equation}
\z_i(t) = |\nabla_{x_i} u(t)|^{p_i - 2} \nabla_{x_i} u(t) \quad \mathcal{L}^N-\hbox{a.e. in } \Omega \hbox{ for all } i = 2,...,k;
\end{equation}
\begin{equation}  
u_t(t) = \mathrm{div}(\z(t))\quad  \hbox{in} \ \mathcal{D}^\prime(\Omega);
\end{equation}
\begin{equation}
[\z_1(t), \nu_{x_1}] \in \mathrm{sign}(-u(t)) \qquad  \mathcal{H}^{N-1}-\mbox{a.e.  on} \ \partial \Omega_1 \times \Omega_2.
\end{equation}
}
\end{definition}

With this definition, since $\mathcal{A}$ coincides with $\partial\mathcal{F}$, by the Brezis-Komura theorem (Theorem \ref{thm:breziskomura}) we get the following existence and uniqueness result. The comparison principle is a consequence of the complete accretivity of the operator $\mathcal{A}$.

\begin{theorem}
Suppose that $\Omega = \Omega_1 \times \Omega_2$ is a bounded Lipschitz domain, where $\Omega_1 \subset \R^{n_1}$ and $\Omega_2 \subset \R^{n_2 + ... + n_k}$. Then, for every $u_0 \in L^2(\Omega)$ there exists a unique weak solution $u \in C([0,T]; L^2(\Omega)) \cap W_{\rm loc}^{1,2}(0, T; L^2(\Omega))$ to the homogeneous Dirichlet problem \eqref{eq:mainproblem} with initial datum $u_0$. 

Moreover, the following comparison principle holds: for all $r \in [1,\infty]$, if $u_1, u_2$ are weak solutions for the initial data $u_{1,0}, u_{2,0} \in L^2(\Omega) \cap L^r(\Omega)$ respectively, then
\begin{equation}
\Vert (u_1(t) - u_2(t))^+ \Vert_r \leq \Vert ( u_{1,0}- u_{2,0})^+ \Vert_r.
\end{equation}
\end{theorem}

We conclude by giving an equivalent characterisation of weak solutions in terms of an integral equality satisfied on almost every time slice.

\begin{corollary}
Let $u_0 \in L^2(\Omega)$ and assume that $u \in C([0,T]; L^2(\Omega)) \cap W_{\rm loc}^{1,2}(0, T; L^2(\Omega))$ satisfies $u(0, \cdot) =u_0$. Then, $u$ is a weak solution to the Dirichlet problem \eqref{eq:mainproblem} if and only if for almost all $t \in (0,T)$ we have $u(t) \in BV_{0,w}^{(p_i)}(\Omega)$ and there exists a vector field $\mathbf{z} \in X_2^{(p'_i)}(\Omega)$ such that $\| \z_1 \|_\infty \leq 1$, $u_t(t) = \mathrm{div}(\z(t))$ in the sense of distributions and 
\begin{align}\label{eq:equivalentcondition}
\int_\Omega |D_{x_1} u(t)| + \int_{\partial\Omega_1 \times \Omega_2} &|u(t)| \, d\mathcal{H}^{N-1} + \sum_{i = 2}^k \int_{\Omega} |\nabla_{x_i} u(t)|^{p_i} \, dx + \int_{\Omega} u_t(t) (u(t)-v) \, dx \\ 
&= \int_\Omega (\z_1(t), D_{x_1} v) + \sum_{i=2}^k \int_\Omega \z_i \cdot \nabla_{x_i} v \, dx - \int_{\partial\Omega_1 \times \Omega_2} [\z_1, \nu_{x_1}] \, v \, d\mathcal{H}^{N-1}
\end{align}
for every $v \in BV_{0,w}^{(p_i)}(\Omega) \cap L^2(\Omega)$.


\end{corollary}

\begin{proof}
The implication in one direction is very simple: if $u$ is such that equation \eqref{eq:equivalentcondition} holds for a.e. $t$, it is enough to take $v = u(t)$. Then, equation \eqref{eq:equivalentcondition} becomes
\begin{align}
\int_\Omega |D_{x_1} u(t)| &+ \int_{\partial\Omega_1 \times \Omega_2} |u(t)| \, d\mathcal{H}^{N-1} + \sum_{i = 2}^k \int_{\Omega} |\nabla_{x_i} u(t)|^{p_i} \, dx \\
&= \int_\Omega (\z_1(t), D_{x_1} u(t)) + \sum_{i=2}^k \int_\Omega \z_i \cdot \nabla_{x_i} u(t) \, dx - \int_{\partial\Omega_1 \times \Omega_2} [\z_1(t), \nu_{x_1}] \, u(t) \, d\mathcal{H}^{N-1}
\end{align}
Applying Proposition \ref{prop:propertiesofanzellottipairing} and Theorem \ref{thm:definitionofweaktrace}, we get that the vector field $\mathbf{z}$ and the function $u$ satisfy the desired properties.

In the other direction, assume that $u$ is a weak solution to problem \eqref{eq:mainproblem}. Applying the anisotropic Gauss-Green formula (Theorem \ref{thm:anisotropicgaussgreen}), we get that
\begin{align}
\int_\Omega (\z_1, D_{x_1} v) &+ \sum_{i=2}^k \int_\Omega \z_i \cdot \nabla_{x_i} v \, dx - \int_{\partial\Omega_1 \times \Omega_2} [\z_1(t), \nu_{x_1}] \, v \, d\mathcal{H}^{N-1} \\
&= - \int_\Omega v \, \mathrm{div}(\z) \, dx = \int_\Omega (u(t) - v) \, \mathrm{div}(\z) \, dx - \int_\Omega u(t) \, \mathrm{div}(\z) \, dx \\
&= \int_\Omega u_t(t) (u(t) - v) \, dx + \int_\Omega (\z_1, D_{x_1} u(t))  \\
&\qquad\qquad\qquad + \sum_{i=2}^k \int_\Omega \z_i \cdot \nabla_{x_i} u(t) \, dx - \int_{\partial\Omega_1 \times \Omega_2} [\z_1(t), \nu_{x_1}] \, u(t) \, d\mathcal{H}^{N-1} \\
&= \int_\Omega u_t(t) (u(t) - v) \, dx + \int_\Omega |D_{x_1} u(t)| \\
&\qquad\qquad\qquad + \int_{\partial\Omega_1 \times \Omega_2} |u(t)| \, d\mathcal{H}^{N-1} + \sum_{i = 2}^k \int_{\Omega} |\nabla_{x_i} u(t)|^{p_i} \, dx,
\end{align}
which concludes the proof.
\end{proof}

\begin{remark}
In the paper \cite{FVV}, the anisotropic $p$-Laplacian evolution equation is studied on the whole space $\mathbb{R}^N$ in place of a bounded domain with homogeneous Dirichlet data. It is not immediate if the proof above generalises to the case $\Omega = \mathbb{R}^N$ for the following reason: in Definition \ref{dfn:anzellottipairing}, the term
\begin{equation}
-\int_\Omega u \, \varphi \, \mathrm{div}(\z) \, dx - \sum_{i = 1}^k \int_\Omega u \, \z_i \cdot \nabla_{x_i} \varphi \, dx
\end{equation}
is well-defined once we ensure that the products $u \, \mathrm{div}(\z)$ and $u \, \z_i$ are integrable for all $i = 1, ..., k$. For the classical Anzellotti pairing, we assume that $\z \in L^\infty(\mathbb{R}^N; \mathbb{R}^N)$ and the integrability of the second term is not a problem; however, in our case the only information we get from the equation is that $\z_i \in L^{p'_i}(\mathbb{R}^N; \mathbb{R}^N)$ and a joint condition for the divergence, so in order for the second term to be well-defined we need to assume that $u \in L^{p_i}(\mathbb{R}^N)$ for all $i = 1,...,k$. For a bounded domain with homogeneous boundary data, this condition holds due to the Poincar\'e inequality and the embedding $L^{p_k}(\Omega) \subset L^{p_i}(\Omega)$ for all $i = 1,...,k$. However, even if we enforce a condition $u \in L^2(\mathbb{R}^N)$ to make the first term integrable, on an unbounded domain we do not know from the equation if $u$ has sufficient regularity, so in order to define solutions to the anisotropic $p$-Laplace evolution equation one would need to introduce a different version of Anzellotti pairings and adapt the definition accordingly.
\end{remark}

\begin{remark}
With essentially the same proof, we can study the following problem with mixed boundary conditions, in which we have Neumann boundary conditions on the linear-growth coordinates and a homogeneous Dirichlet condition on the superlinear-growth coordinates. To be precise, consider the problem
\begin{equation}\label{eq:mixedproblem}
\left\{ \begin{array}{lll} u_t (t,x) = \Delta_{(p_i)} u   \quad &\hbox{in} \ \ (0, T) \times \Omega; \\[5pt] u(t) = 0 \quad &\hbox{on} \ \ (0, T) \times  (\Omega_1 \times \partial\Omega_2); 
\\[5pt] \frac{\partial u}{\partial \eta} = 0 \quad &\hbox{on} \ \ (0, T) \times  (\partial\Omega_1 \times \Omega_2);
\\[5pt] u(0,x) = u_0(x) \quad & \hbox{in} \ \  \Omega, \end{array} \right.
\end{equation}
where $u_0 \in L^2(\Omega)$. The associated functional $\widetilde{\mathcal{F}}: L^2(\Omega) \rightarrow [0, + \infty]$ is given by
\begin{equation}
\widetilde{\mathcal{F}}(u):= \left\{ \begin{array}{ll} \displaystyle\int_\Omega |D_{x_1} u| + \sum_{i=2}^n \frac{1}{p_i} \int_\Omega |\nabla_{x_i} u|^{p_i} \, dx \quad &\hbox{if} \ u \in BV_{0,w}^{(p_i)}(\Omega) \cap L^2(\Omega); \\ \\ + \infty \quad &\hbox{if} \ u \in  L^2(\Omega) \setminus BV_{0,w}^{(p_i)}(\Omega).
\end{array}\right.
\end{equation}
Then, $\widetilde{\mathcal{F}}$ is convex, lower semicontinuous with respect to convergence in $L^2(\Omega)$, and its subdifferential in $L^2(\Omega)$ can be described as follows: $(u,v) \in \partial\widetilde{\mathcal{F}}$ if and only if $u, v \in L^2(\Omega)$, $u \in BV_{0,w}^{(p_i)}(\Omega)$ and there exists a vector field $\z = (\z_1, ..., \z_n) \in X^{(p'_i)}_2(\Omega)$ such that the following conditions hold:
\begin{equation} 
\| \z_1 \|_\infty \leq 1;
\end{equation}
\begin{equation}
(\z_1, D_{x_1} u) = |D_{x_1} u| \quad \hbox{as measures};
\end{equation}
\begin{equation} 
\z_i = |\nabla_{x_i} u|^{p_i - 2} \nabla_{x_i} u \quad \hbox{for all } i = 2,...,k;
\end{equation}
\begin{equation}
-\mathrm{div}(\z) = v \quad \hbox{in} \ \Omega;
\end{equation}
\begin{equation}
[\z_1, \nu_{x_1}] = 0 \qquad  \mathcal{H}^{N-1}-\mbox{a.e.  on} \ \partial \Omega_1 \times \Omega_2.
\end{equation}
Moreover, $\partial \widetilde{\mathcal{F}}$ is completely accretive and its domain is dense in $L^2(\Omega)$. The proof of this result is very similar to the one given in Theorem \ref{thm:subdifferentialgradientflow}. It employs the same strategy, function spaces, operators and functionals, the only differences being that the functional $E_0$ is identically zero (so the dual functional $E_0^*$ is finite only at $0$, which implies the condition $[\z_1, \nu_{x_1}] = 0$) and that all boundary terms disappear in the computations.

Then, by the Brezis-Komura theorem (Theorem \ref{thm:breziskomura}), we get the following existence and uniqueness result. Then, for every $u_0 \in L^2(\Omega)$, there exists a unique weak solution to problem \eqref{eq:mixedproblem} with initial datum $u_0$ in the following sense: $u \in C([0,T]; L^2(\Omega)) \cap W_{\rm loc}^{1,2}(0, T; L^2(\Omega))$, $u(0, \cdot) =u_0$, and for almost all $t \in (0,T)$ we have $u(t) \in BV_{0,w}^{(p_i)}(\Omega)$ and there exist vector fields $\z(t) \in X_2^{(p'_i)}(\Omega)$ such that the following conditions hold:
\begin{equation}
\| \z_1(t) \|_\infty \leq 1;
\end{equation}
\begin{equation}
(\z_1(t), D_{x_1} u(t)) = |D_{x_1} u(t)| \quad \hbox{as measures};
\end{equation}
\begin{equation}
\z_i(t) = |\nabla_{x_i} u(t)|^{p_i - 2} \nabla_{x_i} u(t) \quad \mathcal{L}^N-\hbox{a.e. in } \Omega \hbox{ for all } i = 2,...,k;
\end{equation}
\begin{equation}  
u_t(t) = \mathrm{div}(\z(t))\quad  \hbox{in} \ \mathcal{D}^\prime(\Omega);
\end{equation}
\begin{equation}
[\z_1(t), \nu_{x_1}] = 0 \qquad  \mathcal{H}^{N-1}-\mbox{a.e.  on} \ \partial \Omega_1 \times \Omega_2.
\end{equation}
Moreover, the following comparison principle holds: for all $r \in [1,\infty]$, if $u_1, u_2$ are weak solutions for the initial data $u_{1,0}, u_{2,0} \in L^2(\Omega) \cap L^r(\Omega)$ respectively, then
\begin{equation}
\Vert (u_1(t) - u_2(t))^+ \Vert_r \leq \Vert ( u_{1,0}- u_{2,0})^+ \Vert_r.
\end{equation}
\end{remark}

\section{The elliptic problem}\label{sec:elliptic}

Finally, we briefly comment on the elliptic problem for the anisotropic $p$-Laplacian operator in the case when on some of the coordinates we have linear growth. It has been previously studied in \cite{MRST} in the case with just two exponents; in our notation, this corresponds to $(p_1,p_2) = (1,q)$. The authors proved existence and characterisation of solutions using an approximation by solutions of the anisotropic $p$-Laplace equation with two exponents $p,q$ and passing to the limit $p \rightarrow 1$. In principle, the results from that paper can be generalised to the full anisotropic $p$-Laplacian operator with a similar technique; here, we instead present a short proof based on the Fenchel-Rockafellar duality theorem (Theorem \ref{thm:fenchelrockafellar}).

We keep the assumptions on the domain and exponents as in the previous Sections. Take $f \in L^{p'_k}(\Omega)$ and consider the following problem
\begin{equation}\label{eq:ellipticproblem}
\left\{ \begin{array}{lll} -\Delta_{(p_i)} u = f   \quad &\hbox{in} \ \ \Omega; \\[5pt] u = 0 \quad &\hbox{on} \ \ \partial\Omega.
\end{array} \right.
\end{equation}
This corresponds to the minimisation of the functional $\mathcal{J}: L^{p_k}(\Omega) \rightarrow \mathbb{R}$ given by
\begin{equation}
\mathcal{J}(u):= \left\{ \begin{array}{ll} \displaystyle 
\int_{\Omega} |D_{x_1} u| + \sum_{i=2}^k \frac{1}{p_i} \int_{\Omega} |\nabla_{x_i} u|^{p_i} \, dx + \int_{\partial\Omega_1 \times \Omega_2} |u| \, d\mathcal{H}^{N-1} - \int_{\Omega} fu \, dx \quad &\hbox{if} \ u \in BV_{0,w}^{(p_i)}(\Omega); \\ \\ + \infty \quad &\hbox{if} \ u \notin  BV_{0,w}^{(p_i)}(\Omega),
\end{array}\right.
\end{equation}
which is lower semicontinuous with respect to convergence in $L^{p_k}(\Omega)$ (recall that $BV_{0,w}^{(p_i)}(\Omega)$ embeds continuously into $L^{p_k}(\Omega)$ by Proposition \ref{prop:poincare}). Clearly, it is also convex, and again using Proposition \ref{prop:poincare} we see that for $u \in BV_{0,w}^{(p_i)}(\Omega)$
\begin{align}
\mathcal{J}(u) &\geq \frac{1}{p_k} \int_{\Omega} |\nabla_{x_k} u|^{p_k} \, dx - \int_{\Omega} fu \, dx \geq C \int_{\Omega} |u|^{p_k} \, dx - \int_{\Omega} fu \, dx \\
&\geq C \| u \|_{p_k}^{p_k} - \| f \|_{p'_k} \| u \|_{p_k} \geq \| u \|_{p_k} (C \| u \|_{p_k}^{p_k - 1} - \| f \|_{p'_k}) \geq - C \| f \|_{p'_k}^{1 + \frac{1}{p_k - 1}},
\end{align}
so $\mathcal{J}$ is bounded from below (even in the absence of any smallness assumptions on $f$). A similar computation shows that $\mathcal{J}$ is coercive: whenever $\mathcal{J}(u) \leq M$, we have
\begin{align}
C \| u \|_{p_k}^{p_k} = C \int_{\Omega} |u|^{p_k} \, dx \leq \frac{1}{p_k} \int_{\Omega} |\nabla_{x_k} u|^{p_k} \, dx \leq M + \int_{\Omega} fu \, dx \leq M + \| f \|_{p'_k} \| u \|_{p_k},
\end{align}
so either $\| u \|_{p_k} \leq 1$ or we divide both sides by $\| u \|_{p_k}$ and get
\begin{align}
C \| u \|_{p_k}^{p_k - 1} \leq M + \| f \|_{p'_k},
\end{align}
so $\mathcal{J}$ is coercive. By the direct method of the calculus of variations, it admits a minimiser.

Our definition of solutions will be based on the characterisation of the subdifferential of $\mathcal{J}$ in $L^{p_k}(\Omega)$. The dual space to $L^{p_k}(\Omega)$ is $L^{p'_k}(\Omega)$; in this duality, for any $u \in L^{p_k}(\Omega)$ we can define the subdifferential of the convex and lower semicontinuous functional $\mathcal{J}$ as follows:
\begin{equation*}
\partial \mathcal{J}(u) := \bigg\{ w \in L^{p'_k}(\Omega): \, \mathcal{J}(v) - \mathcal{J}(u) \geq \int_\Omega w (v - u) \, dx \quad  \forall \, v \in L^{p_k}(\Omega) \bigg\}.
\end{equation*}
Then, provided that $u \in BV_{0,w}^{(p_i)}(\Omega)$, the subdifferential $\partial\mathcal{J}(u)$ is a convex, closed and nonempty set (it is empty when $u \notin BV_{0,w}^{(p_i)}(\Omega)$). Moreover, $u$ is a minimiser of the functional $\mathcal{J}$ if and only if $0 \in \partial\mathcal{J}$, so the Euler-Lagrange equation associated to minimisation of $\mathcal{J}$ is
\begin{equation}\label{eq:eulerlagrangeequation}
0 \in \partial\mathcal{J}(u).
\end{equation}
We now give a precise characterisation of solutions to \eqref{eq:eulerlagrangeequation}.

\begin{definition}\label{dfn:solutionsellipticproblem}
We say that $u \in BV_{0,w}^{(p_i)}(\Omega)$ is a weak solution to problem \eqref{eq:ellipticproblem} if there exists a vector field $\z = (\z_1, ..., \z_n) \in X^{(p'_i)}_{p'_k}(\Omega)$ such that the following conditions hold:
\begin{equation}\label{eq:ellipticbound}
\| \z_1 \|_\infty \leq 1;
\end{equation}
\begin{equation}\label{eq:ellipticconditionforpairing}
(\z_1, D_{x_1} u) = |D_{x_1} u| \quad \hbox{\rm as measures};
\end{equation}
\begin{equation}\label{eq:ellipticconditiononzi}
\z_i = |\nabla_{x_i} u|^{p_i - 2} \nabla_{x_i} u \quad \mathcal{L}^N-\hbox{\rm a.e. in } \Omega \,\,\, \hbox{\rm for all } i = 2,...,k;
\end{equation}
\begin{equation}\label{eq:ellipticconditionfordivergence}
-\mathrm{div}(\z) = f \quad \hbox{\rm in} \ \Omega;
\end{equation}
\begin{equation}\label{eq:ellipticboundarycondition}
[\z_1, \nu_{x_1}] \in \mathrm{sign}(-u) \qquad  \mathcal{H}^{N-1}-\mbox{\rm a.e.  on} \ \partial \Omega_1 \times \Omega_2.
\end{equation}
\end{definition}

Naturally, this definition is similar to the one given in Definition \ref{dfn:operatorad}, the difference being that the function spaces are different and that $f$ takes the role of $v$ due to the additional term in $\mathcal{J}$. Therefore, in order to prove existence of solutions in this sense, we will employ a similar argument to the one in Theorem \ref{thm:subdifferentialgradientflow}, and for this reason we use the same notation for coordinates and skip the proofs of some intermediate claims.

\begin{theorem}\label{thm:existenceellipticproblem}
There exists a solution to problem \eqref{eq:eulerlagrangeequation} in the sense of Definition \ref{dfn:solutionsellipticproblem}.
\end{theorem}

\begin{proof} 
{\bf Step 1.} We present the problem in a framework adapted to the Fenchel-Rockafellar duality theorem (Theorem \ref{thm:fenchelrockafellar}). We restrict our attention to $W_{0,w}^{1,(p_i)}(\Omega)$ and set 
$$U = W_{0,w}^{1,(p_i)}(\Omega); \qquad V = L^1(\partial\Omega_1 \times \Omega_2) \times L^1(\Omega; \mathbb{R}^{n_1}) \times ... \times L^{p_n}(\Omega; \mathbb{R}^{n_k}). $$
The linear and continuous operator $A: U \rightarrow V$ is defined as
$$ Au = (u|_{\partial\Omega_1 \times \Omega_2}, \nabla_{x_1} u, ..., \nabla_{x_k} u)$$
and we set $E: V \rightarrow \mathbb{R}$ by the formula
\begin{equation}
E(v_0,v_1,...,v_k) = E_0(v_0) + \sum_{i=1}^k E_i(v_i),
\end{equation}
where
\begin{equation}
E_0(v_0) = \int_{\partial\Omega_1 \times \Omega_2} \vert v_0 \vert \,d\mathcal{H}^{N-1}
\end{equation}
and
\begin{equation}
E_i(v_i) = \frac{1}{p_i} \int_\Omega |v_i|^{p_i} \, dx.
\end{equation}
Then, $E$ is proper, convex and lower semicontinuous. We also set $G: W_{0,w}^{1,(p_i)}(\Omega) \rightarrow \mathbb{R}$ by
$$G(u):= - \int_\Omega f u \, dx$$
and see that it a proper, convex, and continuous functional.

{\flushleft \bf Step 2.} We now compute the convex conjugates of $E$ and $G$. The functionals $E_i^*$ are given by the same formulas as in the proof of Theorem \ref{thm:subdifferentialgradientflow}, i.e.
\begin{equation}
E_0^*(v_0^*) = \left\{ \begin{array}{ll} \displaystyle  0 \quad &\hbox{if} \ \  |v_0^*| \leq 1 \quad \mathcal{H}^{N-1}-\mbox{a.e. on } \partial\Omega_1 \times \Omega_2; \\[10pt] +\infty \quad &\hbox{otherwise,}   \end{array}  \right.
\end{equation}
for $i = 1$ we have
$$E_1^*(v_1^*) = \twopartdef{0}{\mbox{if } \| v_1^* \|_\infty \leq 1;}{+\infty}{\mbox{otherwise,}} $$
and for $i = 2, ..., k$ 
$$ E_i^*(v_i^*) = \frac{1}{p'_i} \int_\Omega |v_i^*|^{p'_i} \, dx.$$
Note that $A^* v^* =- \mathrm{div}((v_1^*, ..., v_k^*))$; in particular, $(v_1^*, ..., v_k^*) \in X_{p'_k}^{(p'_i)}(\Omega)$,
$$G^*(A^* v^*) = \left\{ \begin{array}{ll} \displaystyle  0 \quad &\hbox{if} \ \  \mathrm{div}((v_1^*, ..., v_k^*)) = f; \\[10pt] +\infty \quad &\hbox{otherwise,}   \end{array}  \right. $$
and we conclude as in the proof of Theorem \ref{thm:subdifferentialgradientflow} that $v_0^* = - [v_1^*, \nu_{x_1}]$.

{\flushleft \bf Step 3.} Now, consider the minimisation problem
\begin{equation}
\inf_{u \in U} \bigg\{ E(Au) + G(u) \bigg\}
\end{equation}
with $E$ and $G$ defined as above. For $u_0 \equiv 0$ we have $E(Au_0) = G(u_0) = 0 < \infty$ and $E$ is continuous at $0$. Then, Theorem \ref{thm:fenchelrockafellar} implies that the dual problem given by
\begin{equation}
\sup_{v^* \in V^*} \bigg\{ - E^*(-v^*) - G^*(A^* v^*) \bigg\}
\end{equation}
admits at least one solution and there is no duality gap, i.e. the infimum in the first problem is equal to the supremum in the second one. We first split the dual problem into coordinates and write
\begin{equation}
\sup_{v^* \in L^\infty(\partial\Omega_1 \times \Omega_2) \times L^\infty(\Omega; \mathbb{R}^{n_1}) \times ... \times L^{p'_k}(\Omega; \mathbb{R}^{n_k})} \bigg\{ - E_0^*(-v_0^*) - \sum_{i=1}^k E_i^*(-v_i^*) - G^*(A^* v^*) \bigg\}.
\end{equation}
Recall that if $G^*(A^* v^*)$ is finite, we necessarily have that $(v_1^*, ..., v_k^*) \in X_{p'_k}^{(p'_i)}(\Omega)$. Moreover, observe that $E_0^*$, $E_1^*$ and $G$ take only values $0$ and $+\infty$. From now on, denote by $\mathcal{Z}$ the subset of $V^*$ such that the dual problem does not immediately return $-\infty$, i.e.
\begin{align}
\mathcal{Z} = \bigg\{ v^* \in &L^\infty(\partial\Omega_1 \times \Omega_2) \times X_{p'_k}^{(p'_i)}(\Omega): \quad |v_1^*| \leq 1 \quad \mathcal{L}^N-\mbox{a.e in } \Omega; \\
& |v_0^*| \leq 1 \quad \mathcal{H}^{N-1}-\mbox{a.e. on } \partial\Omega_1 \times \Omega_2; \quad \mathrm{div}((v_1^*, ..., v_k^*)) = f; \quad v_0^* = -[v_1^*, \nu_{x_1}] \bigg\}.
\end{align}
Using the newly defined set $\mathcal{Z}$ and the explicit formulas for $E_i^*$ and $G$, we can rewrite the dual problem as
\begin{equation}\label{eq:dualelliptic}
\sup_{v^* \in \mathcal{Z}} \bigg\{ - \sum_{i = 2}^k \frac{1}{p'_i} \int_\Omega |v_i^*|^{p'_i} \, dx \bigg\}.
\end{equation}
Due to the constraint $v_0^* = - [v_1^*, \nu_{x_1}]$, the first coordinate of any solution to the dual problem is determined by the other coordinates. Therefore, if we set
$$
\mathcal{Z}' = \bigg\{ (v_1^*,...,v_k^*) \in  X_{p'_k}^{(p'_i)}(\Omega): \quad |v_1^*| \leq 1 \quad \mathcal{L}^N-\mbox{a.e in } \Omega; \quad \mathrm{div}((v_1^*, ..., v_k^*)) = f \bigg\},
$$
we may further simplify the dual problem to
\begin{equation}\label{eq:simplifieddual}
\sup_{(v_1^*,...,v_k^*) \in \mathcal{Z}'} \bigg\{ - \sum_{i = 2}^k \frac{1}{p'_i} \int_\Omega |v_i^*|^{p'_i} \, dx \bigg\}.
\end{equation}

{\flushleft \bf Step 4.} Finally, we show that whenever $v^*$ is a solution to the dual problem \eqref{eq:dualelliptic} and $u \in BV_{0,w}^{(p_i)}(\Omega)$ is a minimiser of the functional $\mathcal{J}$, then $\mathbf{z} = -(v_1^*,...,v_k^*)$ satisfies the conditions given in Definition \ref{dfn:solutionsellipticproblem}. It is clear that the divergence constraint \eqref{eq:ellipticconditionfordivergence} is satisfied and that $\| \z_1 \|_\infty \leq 1$. By the first subdifferentiability property \eqref{eq:epsilonsubdiff1}, we have
\begin{align}
0 \leq \int_{\partial\Omega_1 \times \Omega_2} |u_n| \, d\mathcal{H}^{N-1} &+ \sum_{i=1}^k \frac{1}{p_i} \int_\Omega |\nabla_{x_i} u_n|^{p_i} \, dx + \sum_{i=2}^k \frac{1}{p'_i} \int_\Omega |v_i^*|^{p'_i} \, dx \\
&+ \int_{\partial\Omega_1 \times \Omega_2} v_0^* u_n \, d\mathcal{H}^{N-1} + \sum_{i=1}^k \int_\Omega v_i^* \cdot \nabla_{x_i} u_n \, dx  \leq \varepsilon_n.
\end{align}
Since we assumed that the trace of $u_n$ on $\partial\Omega_1 \times \Omega_2$ is fixed and equal to the trace of $u$, the boundary integral does not change with $n$, so it is equal to zero. Therefore,
\begin{equation}
v_0^* \in \mathrm{sign}(-u) \qquad  \mathcal{H}^{N-1}-\mbox{a.e.  on} \ \partial \Omega_1 \times \Omega_2,
\end{equation}
and since $v_0^* = [-v_1^*, \nu_{x_1}]$, the condition \eqref{eq:ellipticboundarycondition} holds for $\mathbf{z} = -(v_1^*,...,v_k^*)$. Once we take into account that the boundary terms disappear, the first subdifferentiability property \eqref{eq:epsilonsubdiff1} yields
\begin{align}\label{eq:beforepassingtothelimitelliptic}
0 \leq \sum_{i=1}^k \frac{1}{p_i} \int_\Omega |\nabla_{x_i} u_n|^{p_i} \, dx + \sum_{i=2}^k \frac{1}{p'_i} \int_\Omega |v_i^*|^{p'_i} \, dx + \sum_{i=1}^k \int_\Omega v_i^* \cdot \nabla_{x_i} u_n \, dx  \leq \varepsilon_n,
\end{align}
and making a similar computation as in the proof of Theorem \ref{thm:subdifferentialgradientflow} we conclude that
\begin{align}
\int_\Omega |D_{x_1} u| + \int_{\Omega} (v_1^*, D_{x_1} u) \geq 0
\end{align}
and for each $i = 2,...,k$ we have a pointwise inequality
\begin{align}
\frac{1}{p_i} |\nabla_{x_i} u|^{p_i} + \frac{1}{p'_i} |v_i^*|^{p'_i} + v_i^* \cdot \nabla_{x_i} u \geq 0
\end{align}
almost everywhere in $\Omega$. Hence, the two inequalities above need to be equalities, so properties \eqref{eq:ellipticconditionforpairing} and \eqref{eq:ellipticconditiononzi} respectively hold for the choice $\z = -(v_1^*,...,v_k^*)$. Therefore, we proved that all the conditions \eqref{eq:ellipticbound}-\eqref{eq:ellipticboundarycondition} in Definition \ref{dfn:solutionsellipticproblem} hold and consequently $u$ is a solution to problem \eqref{eq:ellipticproblem}.
\end{proof}

\begin{remark}
Using a similar argument, it is possible to prove existence of solutions to problem \eqref{eq:ellipticproblem} in the sense of Definition \ref{dfn:solutionsellipticproblem} whenever $f \in L^{\overline{p}'}(\Omega)$ and the domain $\Omega$ is such that there is an embedding $BV_{0,w}^{(p_i)}(\Omega) \hookrightarrow L^{\overline{p}}(\Omega)$. We skip the proof as it is almost identical to the one above. Note that the exponent $\overline{p}'$ coincides with the exponent given in \cite{MRST} for problem \eqref{eq:introductionellipticproblem}.
\end{remark}

\bigskip

\noindent {\bf Acknowledgments.} The author would like to thank prof. Segura de Le\'on for fruitful discussions concerning this paper. This research was funded partially by the Austrian Science Fund (FWF), grant ESP 88. The author has also been partially supported by the OeAD-WTZ project CZ 01/2021. For the purpose of open access, the author has applied a CC BY public copyright licence to any Author Accepted Manuscript version arising from this submission.

\end{document}